\documentclass{article}
\usepackage{amsmath, amssymb, amsthm}
\usepackage{graphicx}
\usepackage{subfig}

\newtheorem{theorem}{Theorem}
\newtheorem{definition}{Definition}

\newtheorem{lemma}{Lemma}

\newcommand{\ie}{{\em i.\thinspace{}e. }}
\newcommand{\etal}{{\em et al. }}

\title{Sobolev gradients and image interpolation}

\author{Parimah Kazemi (parimah.kazemi@gmail.com) and Ionut Danaila}
\begin{document}
\maketitle

\begin{abstract}
We present here a new image inpainting algorithm based on the Sobolev gradient method in conjunction with the Navier-Stokes model.  The original model of Bertalm{\'\i}o \etal [Proc. Conf. Comp. Vision Pattern Rec., 2001] is reformulated as a variational principle based on the minimization of a well chosen functional by a steepest descent method using Sobolev gradients. This new theoretical framework offers an alternative of the direct solving of a high-order partial differential equation, with the practical advantage of an easier and more flexible computer implementation. In particular, the proposed algorithm does not require any constant tuning, nor advanced knowledge of numerical methods for Navier-Stokes equations (slope limiters, dynamic relaxation for Poisson equation, anisotropic diffusion steps, etc). Using a straightforward finite difference implementation, we demonstrate, through various examples for image inpainting and image interpolation, that the novel algorithm is faster than the original implementation of the Navier-Stokes model, while providing results of similar quality.\\
The paper also provides the mathematical theory for the analysis of the algorithm. Using an evolution equation in an infinite dimensional setting we obtain global existence and uniqueness results as well as the existence of an $\omega$-limit. This formalism is of more general interest and could be applied to other image processing models based on variational formulations.
\end{abstract}

\section{Introduction}

Image inpainting refers to the process of filling in an occluded region in an image so that the seam between the original image and the inpainted region is undetectable by a typical viewer.  Some examples of applications are restoring old photographs, removing unwanted objects from images such as text, or removing certain features of an image such as eye glasses.  Image inpainting has been traditionally the work of professional artists.  However, in recent years, much work has been done in the area of digital image inpainting.  The idea of digital image inpainting is that the algorithm, using information from elsewhere in the digital image, fills in the unknown region so that a typical viewer cannot judge what parts of the image are from the original image and what parts are painted in.  In general, the area to be altered consists of many pixels, complex geometries, and possibly spans the entire image.  This establishes image inpainting as a challenging computational problem.

Currently there are a number of models and algorithms that are used in digital inpainting.  The pioneering works in digital image inpainting include \cite{bertalmio2000}, \cite{chan2000}, \cite{masnou1998}, and \cite{nitzberg1993}. In \cite{nitzberg1993}, the authors present a method where the region to be inpainted is viewed as an occluded area.  Then one connects T-junctions with the same grayscale values at the boundary using elastica minimizing curves.  In \cite{masnou1998}, the authors extend on this idea by using a variational formulation.  They inpaint the region with unknown pixel values by connecting isophotes at the boundary using geodesic curves. In \cite{chan2000}, the authors use a geometric model associated with Euler's elastica.  The idea is to view the image as a surface and to minimize a linear combination of the curvature and arc length. More recently, the work in \cite {bornemann}, \cite{criminsi}, and \cite{tschump} provide state of the art results in the area of digital image inpainting.  In \cite{chanbook} and \cite{survey} a comprehensive survey of modern methods is given.

We focus here on the model proposed by Bertalm{\'\i}o, Bertozi and Sapiro (hereinafter BBS), as presented in \cite{bertalmio2001a}. The model is based upon a transport equation analogous to  the transport of vorticity in an incompressible fluid.  The resulting model is a Navier-Stokes type time evolution.  The approximated steady state solution of this model gives the corrected image.  This model has several attractive features.  It is able to fill in regions surrounded by different backgrounds as well as regions that cross through boundaries, it is capable of handling arbitrary topologies, and as a third order PDE, it forces continuity conditions on both the image intensity as well as the gradient of the intensity across the boundary.  The main drawback of the method is that it requires a large number of iterations (several thousands) to converge, making it not very competitive against noniterative fast-marching methods (e. g. \cite{bornemann}).

In this paper, drawing motivation from the model presented in \cite{bertalmio2001a}, we propose to solve the third order PDE by using a different formalism, based on the minimization of the corresponding least-squares derived functional. The idea is to use for the minimization
a steepest descent method with Sobolev gradients, a method that we found  effective in reducing the computational time for different applications (e. g. \cite{ourSISC}). The novelty and the challenge of the present approach is that we deal with a third order PDE, which is, in general,  mathematically and computationally difficult to tackle. We establish in this contribution the mathematical framework for the analysis of the method and show that the numerical performance is considerably improved compared to the original implementation of the model described in  \cite{bertalmio2001a}. We thus offer a new possibility for using the Navier-Stokes model as an alternative to fast-marching methods, that remain unbeatable with regard to computing performance, but are still technical since they  need a careful tuning of constants or functional parameters (see also \cite{survey}).

In the method of calculus of variations, one often solves minimization problems by computing the Euler-Lagrange equations which are used as a direction of descent for the energy.  Explicit iterative schemes using the Euler-Lagrange equations require many iterations to converge due to a tight restriction on the step size resulting from the CFL stability condition. In our minimization scheme, we compute a gradient with respect to a Sobolev metric.  In practice, this amounts to preconditioning the Euler-Lagrange equations for use in minimization.  The preconditioning in general allows for a larger time step when  the resulting evolution equation is discretized in time.  The Sobolev gradient method has proved to be very effective in a variety of image processing applications, for example  \cite{calder2010}, \cite{calder2011}, \cite{renka2009}, and \cite{richardson1}.  The focus of this work is to study the effect of preconditioning that is often associated with the Sobolev gradient, both in terms of the quality of the image and the computational efficiency.  In contrast with the BBS model, we do not use a diffusion term, but instead apply a smoothing operator to the Navier-Stokes gradient flow. An advantage is that our evolution equation is well-posed for all time and we obtain that the minimizer is $C^1$ across the boundary and can be obtained as an $\omega$-limit of the evolution equation.  The minimization of the energy functional with the Sobolev gradient falls into the more general category of image regularization \cite{reg1}, \cite{reg2}, \cite{reg3}, \cite{reg4}, and \cite{reg5} since we expect the interpolated image to be smooth due to continuity properties of the gradient.

Although our results and applications are restricted to solving for the two-dimensi\-onal steady state arrived at by the Navier-Stokes model, our scheme is quite general.  The ideas we present here can easily be adapted to computationally solve a vast variety of higher order partial differential equations.  Since we treat the problem as a variational problem, our scheme can also be applied to existing variational models in image processing (which almost always use the Euler-Lagrange equations for minimization).  An interesting, although not immediate, further application would be the use of the Sobolev gradient method for variational problems such as the ones presented in \cite{arias}.  Finally, our scheme is cast in both an infinite dimensional Hilbert space setting and analogously in a finite dimensional finite--difference setting, making a strong connection between the theory of digital image analysis and applications.

The outline of the paper is as follows.  In section \ref{secnv}, we give a brief explanation of the BBS model as presented in \cite{bertalmio2001a} and make a connection to the Navier-Stokes model for incompressible fluids.  In section \ref{secvar}, we describe the variational formulation from which we obtain our gradient flow in a Hilbert space setting.  We show that the flow has a unique global solution and an $\omega$-limit at which minimization is achieved.  In section \ref{secSG} we obtain an expression for the Sobolev gradient. In section \ref{secdisc}, we give the mathematical formulation in the discrete finite difference setting, and in section \ref{secresults} we present the results from several examples which includes the resulting images as well as a comparison of CPU time, time step size, and the effectiveness of our scheme in solving the boundary value PDE.

\section{The Navier-Stokes model for image inpainting}\label{secnv}

Here we give a summary of the Navier-Stokes model for image inpainting as presented in \cite{bertalmio2001a}, also in \cite{bertalmio2000}, and \cite{bertalmio2001}.  A digital grayscale image can be thought of as a set of $m$ by $n$ data points taking values between 0 and 255.  The resulting function $I$ defined on the $m$ by $n$ grid is called the image intensity.  An isophote of $I$ is a level-line in this context.  The smoothness of the image is represented by the Laplacian of the intensity.  The direction of the level-lines of the intensity $I$ at the $(i,j)^{th}$ grid point is given by the direction that is perpendicular to the gradient of $I$ at the point $(i,j)$.

The idea of the Navier-Stokes model for image inpainting is to propagate the gradient of the smoothness of the images in the direction of the isophotes.  Thus one wants to iterate the evolution equation
\begin{equation}
I'(t)= \nabla^{\perp} I(t) \cdot \nabla \Delta I(t)
\end{equation}
or perhaps, by adding and anisotropic diffusion term,
\begin{equation} \label{anis}
I'(t)= \nabla^{\perp} I(t) \cdot \nabla \Delta I(t) + \nu \nabla \cdot(g|\nabla I(t)| \nabla I(t))
\end{equation}
to steady state solution
\begin{equation}\label{stdy}
\nabla^{\perp} I \cdot \nabla \Delta I=0.
\end{equation}
In equation \eqref{anis}, a second term is added for its smoothing effects, where $\nu >0$ is small and $g$ is a diffusivity function.

The Navier-Stokes model for incompressible Newtonian fluids gives that the velocity field $v$ and pressure $p$ are coupled according to the equation
\begin{equation}
v'(t) + v(t) \cdot \nabla v(t) = - \nabla p(t) + \nu \Delta v(t) \text{ and } \nabla \cdot v(t)=0.
\end{equation}
In two dimensions, the divergence free velocity field possesses a stream function $\psi$ so that $\nabla^{\perp} \psi = v$.  The vorticity can be expressed as $w=\Delta \psi$ and satisfies the advection diffusion equation
\begin{equation}
w'(t)=-v(t) \cdot \nabla w(t)
\end{equation}
when the viscosity $\nu$ is zero.  The steady state must satisfy
\begin{equation}
\nabla^{\perp} \psi \cdot \nabla \Delta \psi=0.
\end{equation}
This is the incompressible Euler equation for fluid flow.  We refer the reader to \cite{temam} for a background on the theory of the Navier-Stokes equation.

Another way to say this is that the gradient of the stream function $\psi$ and the gradient of the Laplacian of the stream function must be parallel.  Thus the analogy between the Navier-Stokes model and image inpainting is that the image intensity acts as the stream function in the Navier-Stokes model.  In our approach, we consider an alternate to the method of Bertalmio as presented in \cite{bertalmio2001a} and minimize the norm of the left hand side  using an evolution equation based on a Sobolev gradient method.

\section{A variational formulation in the continuous setting}\label{secvar}

In this section, we describe our scheme in a Hilbert space setting.  The main idea is that we form an energy functional by taking the norm of the left hand side of equation \eqref{stdy}.  We then search for a minimum of this energy using a gradient flow with a Sobolev gradient.  We start by stating the definitions and results that we will need from the theory of Sobolev spaces.  We then define the variational problem and study properties of the resulting functional.  We argue that every critical point must be a solution of equation \eqref{stdy}.  We define the Sobolev gradient in a general setting and, incorporating the boundary value into the gradient, we obtain a gradient flow that has a unique global solution.  We obtain a critical point of the energy as an $\omega$-limit of the gradient flow.  Since a critical point of the energy corresponds to a zero of the energy, we obtain that our energy functional does indeed posses a minimum.  We discuss regularity properties of the minimizer.

\subsection{Sobolev spaces}\label{secsob}

We review here the basic information regarding the theory of Sobolev spaces.  The information provided here, as well as the notation, are taken from \cite{adams2003}.  Let $\Omega$ be a bounded open subset of $\mathbb{R}^2$.  Recall that
\begin{equation*}
W^{m,p}(\Omega) \equiv \{u \in L^p(\Omega) \ : \ D^{\alpha} u \in L^p(\Omega) \text{ for } 0 \leq |\alpha| \leq m \}
\end{equation*}
where for a multi-index $\alpha$,  $D^{\alpha} u$ denotes the $\alpha^{th}$ partial derivative of $u$ taken in the distributional sense.  Recall also that
\begin{equation*}
H^{m,p}(\Omega) \equiv \text{ the completion of } \{ u \in C^m(\Omega)\ : \ \|u\|_{m,p} < \infty \}
\end{equation*}
where $C^m(\Omega)$ is the space of $m$ times continuously differentiable functions and
\begin{equation*}
\|u\|_{m,p} = \left( \sum_{0 \leq |\alpha| \leq m} \|D^{\alpha} u\|_p^p) \right)^{1/p}
\end{equation*}
with $\| \cdot \|p$ denoting the $L^p$ norm.  $H^{m,p}_0(\Omega)$ is the closure of the the infinitely differentiable functions with compact support in $\Omega$ in $W^{m,p}(\Omega)$.  It is a result (Thm 5.37 of \cite{adams2003}), that under sufficient regularity conditions on the boundary of $\Omega$, the trace of $D^{\alpha} u$ is zero for all $u \in H^{m,p}_0(\Omega)$ and $|\alpha| \leq m$.

It is a result that for every open domain $H^{m,p}(\Omega) = W^{m,p}(\Omega)$.  Since we wish to work in a Hilbert space setting, we fix $p=2$ and refer to the space $H^{m,2}(\Omega)$ as $H^m$, $H^m_0=H^{m,2}_0(\Omega)$, and $L^2(\Omega)=L$.

Let $Du=\{ D^{\alpha} u : 1 \leq |\alpha| \leq k\}$.  We can see that $\left\{ \binom{u}{Du} : u \in H^{m} \right\}$ is a closed subspace of $L(\Omega^m)$ where $\Omega^m=\cup_{|\alpha| \leq m} \Omega_{\alpha}$.  Thus there exists a unique orthogonal projection $P=P(m)$ from $L(\Omega^m)$ onto $\left\{ \binom{u}{Du} : u \in H^{m} \right\}$.  We mention this here as we will need this projection when constructing the Sobolev gradient in section \ref{Sobgrad}.  We also note that $H^m_0$ is a closed subspace of $H^m$.  Thus there exists a unique orthogonal projection of $H^m$ onto $H^m_0$ which we denote by $P_0=P_0(m)$.

Suppose that $\Omega$ is bounded and has $C^m$ boundary $\partial \Omega$.  It is a result, see \cite{adams2003}, that for $u \in H^m$, the boundary traces of $u$ can be defined as $f_{\alpha}(u)=D_{\alpha} u|_{\partial \Omega}$ for $0 \leq |\alpha| \leq m$.  Further $f_{\alpha}(u)$ is in $L^2(\partial \Omega)$.   If $u \in H^m_0$, then $f_{\alpha}(u)=0$ for all $\alpha$.  Finally, we remind the reader the part of the Rellich-Kondrachov theorem that we will need.
\begin{theorem}
Suppose that $\Omega \subset \mathbb{R}^2$ is a bounded domain with a smooth boundary. The embeddings
\begin{equation*}
H^{j+m}(\Omega) \rightarrow C^j(\bar{\Omega})
\end{equation*}
and
\begin{equation*}
H^{j+m} \rightarrow H^j
\end{equation*}
are compact when $m \geq 2$.
\end{theorem}

Here $C^j(\bar{\Omega})$ denotes the $j$ times continuously differentiable functions on $\bar{\Omega}$.

\subsection{Defining the minimization problem}

Assume that $\Omega$ is a bounded open set in $\mathbb{R}^2$ and has smooth boundary.  For $u \in H^3$, define $F$ so that
\begin{equation}\label{f}
F(Du)= \nabla^{\bot} u \cdot \nabla \Delta u.
\end{equation}
Now if $u \in H^3$, then $D^{\alpha} u \in H^2$ when $|\alpha|=1$. Since $H^2$ is compactly embedded in $C(\bar{\Omega})$, for $u \in H^3$, $\nabla u$ is a bounded continuous function, and $|\nabla u | \leq c \|u\|_{H^3}$ for some constant $c$ independent of $u$. $| \cdot |$ denotes the sup norm here and in the rest of the paper, $\| \cdot \|$ denotes the $L^2$ norm, and $\| \cdot\|_{H^m}$ the $H^m$ norm.  Thus
\begin{equation} \label{Fl2}
\|F(Du)\| = \left( \int_{\Omega} |\nabla^{\bot} u \cdot \nabla \Delta u|^2 \right)^{1/2} \leq |\nabla u | \ \| \nabla \Delta u\| \leq c \ \|u\|_{H^3}^2
\end{equation}
and we see that for $u \in H^3$, $F(Du) \in L^2$ and thus the energy
\begin{equation} \label{energy}
E(u)= \frac{\| F(D(u_0)) \|^2}{2}
\end{equation}
is well defined.  To reduce notation, we write $u_0 = P_0 u$, where $P_0$ is the orthogonal projection of $H^3$ onto $H^3_0$ .

In the model of BBS, $u \in H^3$ is sought so that $F(Du_0)=0$ and $u=g$ on the boundary of $\Omega$ for some predefined function $g$.  We assume that $g$ is $C^1$ on the boundary of $\Omega$.

First we show that $E$ is $C^2$ differentiable then explain why a minimizer of $E$ would correspond to a zero of $F \circ D$.

\begin{theorem}\label{Ec2}
$E:H^3 \rightarrow \mathbb{R}$ as defined in equation \eqref{energy} is $C^2$ Fr{\'e}chet differentiable.
\end{theorem}

We compute here the first and second Fr{\'e}chet derivatives of $E$.  We regard the first Fr{\'e}chet derivative of $E$ at $u\in H^3$ as a bounded linear operator on $H^3$ and the second Fr{\'e}chet derivative of $E$ at $u\in H^3$ as a bounded and symmetric bilinear operator on $H^3$.

\begin{proof}
We compute a Fr\'{e}chet derivative of $F$ as given in equation \eqref{f} to obtian
\begin{equation}\label{fprime}
F'(Du)Dh = \nabla^{\perp}h \cdot \nabla \Delta u + \nabla^{\perp}u \cdot \nabla \Delta h.
\end{equation}

One can also compute $F''(Du)(Dh,Dk)$ in order to see that $F \circ D$ is $C^2$ Fr{\'e}chet differentiable on $H^3$.  Next we can compute the first and second Fr{\'e}chet derivatives of $E$ to see that
\begin{equation} \label{Eprime}
E'(u)h=\langle F'(Du_0)Dh_0, F(Du_0) \rangle
\end{equation}
and
\begin{equation*}
E''(u)(h,k) = \langle F'(Du_0)Dh_0, F'(Du_0)Dk_0 \rangle + \langle F''(Du_0)(Dh_0,Dk_0), F(Du_0) \rangle.
\end{equation*}
$E$ being $C^2$ Fr{\'e}chet differentiable follows from the differentiability of $F$.

Notice from equation \eqref{fprime} that $F'(Du)Du = 2F(Du)$.  From equation \eqref{Eprime}, we see that
\begin{equation}\label{Eprimepos}
E'(u)u_0=\langle F'(D u_0)D u_0 , F(Du_0) \rangle = 2 \|F(Du_0)\|^2 = 4E(u).
\end{equation}
\end{proof}

Notice also from equation \eqref{Eprimepos} that $E'(u)u_0 = 2\|F(Du_0)\|^2$.  Hence, in order to solve the boundary value problem, it suffices to find $u \in H^3$ so that $u=g$ on $\partial \Omega$ and $E'(u)h = 0$ for all $h \in C^{\infty}_0(\Omega)$.

\subsection{The Sobolev gradient}

The idea of Sobolev gradients \cite{neuberger2010}, is that given a $C^1$ function $\phi$ that is everywhere defined on a Hilbert space $H$, we can represent the Fr{\'e}chet derivative of $E$ at $u$ using a member of the Hilbert space.  This is due to the Riesz representation theorem as if $\phi$ is $C^1$, then $\phi'(u)$ is a bounded linear functional on $H$.  Thus there exists a unique element of $H$, which we denote by $\nabla E(u)$, so that
\begin{equation}\label{grad}
\phi'(u)h = \langle h , \nabla \phi(u) \rangle_H \text{ for all } h \in H.
\end{equation}
One then considers the gradient flow
\begin{equation}\label{flowgen}
z(0)= x_0 \in H \text{ and } z'(t)=-\nabla \phi(z(t)).
\end{equation}
It follows that $z$ defines a path along which the energy is decreasing.  Assuming one has existence of such a $z$, then it is important to study the asymptotic limit as it is in this limit that we hope to achieve a critical point.

For this problem we consider a gradient with respect to the $H^3$ inner product.  In order to obtain an $\omega$-limit however, we need to map this gradient into a higher order space.  In this section, we give  two results from \cite{neuberger2010} that we will modify for this work and define our evolution equation.

\begin{theorem} \label{SD}
Suppose $\phi$ is a nonnegative valued $C^1$ function on a Hilbert space $H$.  Define the gradient of $\phi$ at $u$ as in equation \eqref{grad}.  If $\nabla \phi$ is a locally Lipschitz function from $H$ to $H$, then for $x_0 \in H$ the gradient flow \eqref{flowgen} has a unique solution for all $t \geq 0$.
\end{theorem}
\begin{theorem} \label{zbounded}
Suppose, under the conditions of theorem \ref{SD}, that $z$ is a global solution of \eqref{flowgen} with energy $\phi$.  Further suppose that $\phi'(u)u \geq 0$ for all $u \in H$, then the range of $z$ is a bounded subset of $H$.
\end{theorem}

The main idea of this theorem is that
\begin{equation*}
(\|z\|_H^2)'(t) = \langle z(t) , z'(t) \rangle_H = -\langle z(t) , \nabla_H \phi (z(t)) \rangle_H = - \phi'(z(t))z(t) \leq 0
\end{equation*}
by assumption.  This implies that $t \rightarrow \|z(t)\|_H^2$ is non increasing and thus the range of $z$ is bounded in $H$.

Since $E$, as defined in equation \eqref{energy}, is Fr{\'e}chet differentiable, $E'(u)$ is a bounded linear functional on $H^3$.  Thus for each $u \in H^3$, there exists a unique member of $H^3$  so that
\begin{equation*}
E'(u)h = \langle h , \nabla_{H^3} E(u) \rangle_{H^3}.
\end{equation*}
For $x \in H^3_0$ and $k >3$, consider the mapping from $H^k_0$ to $\mathbb{R}$ given by
\begin{equation*}
y \rightarrow \langle y , x \rangle_{H^3}.
\end{equation*}

One can see that this is a bounded linear mapping. Thus there exists a unique element $M_{k,0} \ x \in H^k_0$ so that
\begin{equation*}
\langle y , x \rangle_{H^3} = \langle y , M_{k,0} \ x \rangle_{H^k}
\end{equation*}
for all $y \in H^k_0$.  Now consider the operator $M_{k,0}$.  In \cite{kazemi2008}, we obtained that this operator is bounded from $H^3_0$ to $H^k_0$ with norm less than or equal to one.

We define the gradient flow
\begin{equation}\label{flow}
z(0)=u_0 \text{ and } z'(t) = - M_{k,0} P_0 \nabla_{H^3} E(z(t)).
\end{equation}
where $u_0 \in H^k$ and $u_0=g$ on $\partial \Omega$.  Here $P_0$ is the projection of $H^3$ onto $H^3_0$.

\subsection{Global existence, uniqueness, and asymptotic convergence}

In this section we give results regarding the global existence and uniqueness of the gradient flow \eqref{flow}.  We show that a unique solution exists for this system and that the range of $z$ is bounded in $H^k$.  This allows us to extract a subsequence from the range of $z$ which converges to a zero of the gradient thus giving  the existence of a stationary point and a solution to the minimization problem.  We follow the developments in chapter 4 of \cite{neuberger2010} closely.

\begin{theorem} \label{global}
The gradient flow given in equation \eqref{flow} has a unique global solution $z \in C^1([0,\infty) , H^k)$.
\end{theorem}

We follow the proof of theorem \ref{SD} presented in \cite{neuberger2010} with only a slight modification to account for the presence of $M_{k,0} P_0$ in our gradient flow.

\begin{proof}
In order to show that the gradient flow given in \eqref{flow} has a local solution, we need to show that the gradient is locally Lipschtiz from $H^k$ to $H^k$.  We showed in theorem \ref{Ec2} that $E$ is $C^2$ on $H^3$.  Note that for any $C^2$ function $\phi$ defined on a Hilbert space $H$, with gradient as in \eqref{grad}, there exists a constant $c_u$ and a ball $B$ about $u$ so that if $v \in B$ then
\begin{eqnarray*}
|\langle h ,\nabla \phi(u) - \nabla \phi(v)\rangle_H| = |(\phi'(u) - \phi'(v))h| \leq
c_u |u- v|_H |h|_H.
\end{eqnarray*}

Thus if we take $h=\nabla \phi(u) - \nabla \phi(v)$, then we see that
\begin{equation*}
|\nabla \phi(u) - \nabla \phi(v)|_H \leq c_u |u-v|_H.
\end{equation*}

Hence if $\phi$ is a $C^2$ function defined on a Hilbert space, then the gradient of $\phi$ as defined in \eqref{grad} is locally Lipschitz.  From this we conclude that $\nabla_3 E :H^3 \rightarrow H^3$ is locally Lipschitz.  $M_{k,0} P_0 \nabla_{H^3} E$ is locally Lipschitz from $H^k$ to $H^k$ as for $u, \ v \in H^k$,
\begin{eqnarray*}
\|M_{k,0} P_0 (\nabla_{H^3} E(u) -  \nabla_{H^3} E(v)) \|_{H^k} \leq \| P_0 (\nabla_{H^3} E(u) -  \nabla_{H^3} E(v))\|_{H^3} \leq \\
\| (\nabla_{H^3} E(u) -  \nabla_{H^3} E(u)) \|_{H^3} \leq c_u \|u - v \|_{H^3} \leq c_u \|u - v\|_{H^k}.
\end{eqnarray*}
We conclude that the system \eqref{flow} has a local solution.

Let $T$ be the largest number so that a solution exists on $[0,T)$.  In order to show that the flow has a global solution, it suffices to show that $\lim_{t \rightarrow T^-} z(t)$ exists.  This holds true provided that there is a constant $m$ so that
\begin{equation}\label{zlocalbnd}
\|z(a) - z(b)\|_{H^k} \leq m
\end{equation}
for all $0< a, b < T$.  Recalling that $P_0$ is symmetric with respect to $\langle \cdot , \cdot \rangle_{H^3}$ and $z'(t) \in H^k_0$ and hence in the range of $P_0$ for all $t$, we have that
\begin{eqnarray*}
\|z(a) - z(b)\|_{H^k}^2 \leq \left( \int_a^b \|z'\|_{H^k} \right)^2 \leq \text{ (using Cauchy-Schwarz)}\\
(b-a) \int_a^b \|z'\|_{H^k}^2 = (b-a) \int_a^b \langle z' ,z' \rangle_{H^k} = -(b-a) \int_a^b \langle z' , M_{k,0} P_0 \nabla_{H^3} E(z) \rangle_{H^k}\leq \\
-T \int_a^b \langle z' ,P_0 \nabla_{H^3} E(z(t)) \rangle_{H^3} =  -T \int_a^b \langle P_0 z' , \nabla_{H^3} E(z(t)) \rangle_{H^3}=\\
-T \int_a^b \langle z' , \nabla_{H^3} E(z(t)) \rangle_{H^3} = - T \int_a^b E'(z(t))z'(t)  =
- T \int_a^b (E \circ z)' \leq T E(z(0)).
\end{eqnarray*}
Thus equation \eqref{zlocalbnd} holds true and we have that the flow \eqref{flow} has a global solution.  Uniqueness follows from the fact that the projected gradient is a locally Lipschitz function and the fundamental theory for existence and uniqueness of ODE's.
\end{proof}

We show next that the flow \eqref{flow} has an $\omega$-limit which is a minimizer of the energy and a solution to the boundary value problem.

\begin{lemma}\label{grad0}
There exists an unbounded sequence of numbers $\{t_n\}_{n \geq 1}$ so that $M_{k,0} P_0 \nabla_{H^k} E(z(t_n))$ converges to zero in $H^k$.
\end{lemma}

\begin{proof}
Using the same analysis as in the previous proof, we note that
\begin{eqnarray*}
\int_0^{\infty} \| M_{k,0} P_0 \nabla_{H^3} E (z) \|_{H^k}^2 \leq -\int_0^{\infty} (E \circ z)' \leq E(z(0)).
\end{eqnarray*}
From this it follows that
\begin{equation*}
\int_0^{\infty} \| M_{k,0} P_0 \nabla_{H^3} E (z) \|_{H^k}^2
\end{equation*}
is bounded and hence the conclusion follows.
\end{proof}

\begin{lemma}\label{bound}
The range of $z$ is bounded as a subset of $H^k$.
\end{lemma}

The proof is similar to the proof of theorem \ref{zbounded} given in \cite{neuberger2010} with only a slight modification.

\begin{proof}
Let $h(t) = .5 \|z(t)\|_{H^k}^2$ and note that
\begin{eqnarray*}
h'(t) = \langle z(t) , z'(t) \rangle_{H^k} = - \langle z(t) , M_{k,0} P_0 \nabla_{H^3} E(z(t)) \rangle_{H^k} =\\
- \langle P_0 z(t) , \nabla_{H^3} E(z(t)) \rangle_{H^3} = -E'(z(t)) P_0 z(t) = - 4 E(z(t))
\end{eqnarray*}
the last equality following from \eqref{Eprimepos}.  Thus $h'$ is never positive and hence $h$ is bounded.
\end{proof}
\begin{theorem}
Suppose $k \geq 5$ in equation \eqref{flow}.  Then there exists a sequence of unbounded numbers $\{t_n\}_{n \geq 1}$ so that $\{z(t_n)\}_{n \geq 1}$ converges in $H^3$ to $u \in C^1(\bar{\Omega})$ and $E'(u)h=0$ for all $h \in C^{\infty}_0 (\Omega)$.
\end{theorem}

\begin{proof}
Since the range of $z$ is bounded in $H^k$ for $k \geq 3$, using the compact embedding of $H^k$ into $H^3$, there exists a sequence $\{t_n\}_{n \geq 1}$ so that $\{z(t_n)\}_{n \geq 1}$ converges to $u \in H^3$.

Using theorem \ref{grad0}, we can assume $\{t_n\}_{n \geq 1}$ is chosen so that $\{M_{k,0} P_0 \nabla_{H^3} E(z(t_n)) \}_{n \geq 1}$ converges to zero in $H^k$.  Using the regularity properties derived in theorem \ref{global}, $M_{k,0} P_0 \nabla_{H^3} E(z(t_n))$ converges to 0 in $H^k$.  For $h \in C^{\infty}_0$,
\begin{eqnarray*}
E'(u)h = \langle h , \nabla_{H^3} E(u) \rangle_{H^3} =\\
 \langle P_0 h , \nabla_{H^3} E(u) \rangle_{H^3} = \langle h , P_0 \nabla_{H^3} E(u) \rangle_{H^3}=\\
 \langle h ,M_{k,0} P_0 \nabla_{H^3} E(u) \rangle_{H^k}=0.
\end{eqnarray*}
Since $H^3$ is embedded in $C^1(\bar{\Omega})$, $u \in C^1(\bar{\Omega})$.
\end{proof}

We also note that $z(t)$ agrees with $g$ on the boundary of $\Omega$ for all $t$.  Recall that $z$ has range in $H^k$ which is compactly embedded in $C^1(\bar{\Omega})$.  Let $f(t) = z(t) |_{ \partial \Omega}$.  Then $f$ is $C^1$ and $|f'(t)| = |M_{k,0} P_0 \nabla_{H^3} E(z(t)) |_{ \partial \Omega} |=0$.  Since
\begin{equation*}
|f(t) - g| = |f(t) - f(0)| \leq \int_0^t |f'| =0,
\end{equation*}
the assertion follows.   The norm $|\cdot |$ here indicates the sup norm.

Regarding the uniqueness of the solution of the boundary value problem
\begin{eqnarray*}
\nabla^{\bot}u \cdot \nabla \Delta u = 0 \text{ in } \Omega \text{ and }
u=g \text{ on } \partial \Omega,
\end{eqnarray*}
we refer the reader to \cite{bertalmio2001a} where an example is given to illustrate that the solution of the boundary value problem is not unique.  It is also argued that some level of non uniqueness is required for the inpainting problem as one would like to be able to 'choose' the best interpolation for the inpainting problem.

\section{An expression for the gradient}\label{secSG}

Here we give an expression for the Sobolev gradient.  We first derive the Euler-Lagrange equation for our energy functional.  We observe that the Euler-Lagrange equation can be viewed as a  gradient of the energy obtained with respect to an $L^2$ inner product.  We then obtain an expression for the Sobolev gradient using the projection $P$ discussed in section \ref{secsob} and compare the Sobolev gradient with the Euler-Lagrange equation making a key distinction between the two by comparing the regularity properties of the resulting gradients.

\subsection{The Euler-Lagrange equation for $E$}

Recall the expression for $E'(u)$ as given by equation \eqref{Eprime}.  In order to obtain the Euler-Lagrange equation, we need first to find some vector $\vec{v}$ whose components are in $L^2$ so that
\begin{equation}\label{fprimeadj}
\langle F'(Du) Dh , F(Du) \rangle = \langle Dh , \vec{v} \rangle.
\end{equation}
For this purpose, let $y \in L^2$ and recall the expression for $F'(Du)Dh$ as given by equation \eqref{fprime} to obtain
\begin{eqnarray*}
\langle F'(Du)Dh , y \rangle =
\langle \nabla^{\perp}h \cdot \nabla \Delta u + \nabla^{\perp}u \cdot \nabla \Delta h , y \rangle.
\end{eqnarray*}
Since $\nabla^{\perp}h \cdot \nabla \Delta u = - \nabla h \cdot \nabla^{\perp} \Delta u$, we have that
\begin{equation*}
\langle F'(Du)Dh , y \rangle = \langle \nabla h , -y \cdot \nabla^{\perp} \Delta u \rangle + \langle \nabla \Delta h , y \cdot \nabla^{\perp}u \rangle.
\end{equation*}
Here, when $c$ is a scalar, $c \cdot \binom{x}{y}= \binom{cx}{cy}$.

From this calculation we have that the nonzero components of $\vec{v}$ are $v_{\alpha_1}= - F(Du) \cdot \nabla^{\perp} \Delta u$ and $v_{\alpha_2}= F(Du) \cdot \nabla^{\perp}u$ so that
\begin{equation*}
\langle Dh, F(Du) \rangle = \langle \nabla h , v_{\alpha_1} \rangle + \langle \nabla \Delta h ,v_{\alpha_2} \rangle.
\end{equation*}
With this notation, under further regularity assumptions on $u$, the Euler-Lagrange equation for $E$ becomes
\begin{equation}\label{eulerlagrange}
\nabla_{EL} E(u)=D^* \vec{v}
\end{equation}
where $D^*$ denotes the adjoint of $D$ as a closed and densely defined operator. Note that for all $h \in L^2$
\begin{equation*}
E'(u)h=\langle h , \nabla_{EL}E(u) \rangle.
\end{equation*}
Thus the Euler-Lagrange equation, when defined, can be viewed as an $L^2$ gradient for $E$.
\subsection{An orthogonal projection}\label{Sobgrad}

Recall from section \ref{secsob} that $P$ is the orthogonal projection of $L^2(\Omega^k)$ onto $\{\binom{h}{Dh} : h \in H^k\}$ with $Dh$ denoting all partial derivatives of $h$ up to order $k$.  Recall also that the expression for $E'(u)$ as given by equations \eqref{fprime} and \eqref{Eprime}.  Then we have the following
\begin{eqnarray*}
E'(u)h = \langle F'(Du)Dh , F(Du) \rangle =\\
\langle Dh, \vec{v} \rangle =
\langle \binom{h}{Dh}, \binom{0}{\vec{v}} \rangle=\\
\langle P \binom{h}{Dh} , \binom{0}{\vec{v}} \rangle =\\
\langle  Dh ,  P \binom{0}{\vec{v}} \rangle =
\langle h , \Pi P \binom{0}{\vec{v}} \rangle_{H^k}
\end{eqnarray*}
where $\Pi \binom{x}{y} =x$. An expression for $P$, given in \cite{neuberger2010}, is
\begin{equation*}
P_ = \begin{pmatrix}
(I+D^* D)^{-1} & D^*(I+D D^*)^{-1} \\
D(I+D^* D)^{-1} & I-(I+D D^*)^{-1}
\end{pmatrix}
\end{equation*}
where $D^*$ denotes the adjoint of $D$ when viewed as a closed and densely defined linear operator on $L^2(\Omega)$.
From this formula we see that
\begin{equation}\label{gradE}
\nabla_{H^k} E(u)= P \binom{0}{\vec{v}}= D^*(I+D D^*)^{-1}\vec{v}.
\end{equation}
We note that if $u$ satisfies additional regularity so that $\vec{v}$ is in the domain of $D^*$, then
\begin{equation}\label{gradE1}
D^*(I+D D^*)^{-1}\vec{v} = (I+D^* D)^{-1}D^*\vec{v}.
\end{equation}
Thus the Sobolev gradient can be viewed as a smoother gradient than the Euler-Lagrange equation which is the descent direction used in most optimization problems.  In case $k=1$, $(I + D^*D)^{-1}$ is the resolvent of the Neumann Laplacian $(I - \Delta)^{-1}$.  For $k > 1$, a nice argument is given in \cite{calder2011} to show that $(I + D^*D)^{-1}$ is equivalent to $(I - \Delta)^{-k}$.  Thus we see that using a gradient in a higher order Sobolev space is equivalent to preconditioning the Euler-Lagrange equations using the resolvent of the Neumann Laplacian as the preconditioner. A detailed discussion on the preconditioning resulting from the use of gradients in high order Sobolev spaces is presented in the Appendix  \ref{conditioning}.

\section{Finite dimensional implementation} \label{secdisc}

We discretize, using finite differences, equation \eqref{gradE} or equivalently \eqref{gradE1} since in  a finite dimension space the two expressions are equal.  Here $u$ denotes the image intensity, a real valued function defined on a  rectangular domain $R$ of $\mathbb{R}^2$.

\subsection{The inpainting domain}

The inpainting domain is specified by the user.  This can be done by marking the regions to be inpainted using a photo editing software such as Gimp or Paint.  We then determine the inpainting region, $\Omega$, and the inpainting region plus its boundary, $\Omega'$.  We consider $u(i,j)$ to be in $\Omega'$ if $u(k,l) \in \Omega$ for some $(k,l)$ with $0 \leq |i-k| + |j-l| \leq 2$.  We label $u_0$ to be $u$ restricted to $\Omega$ and $u'$ to be $u$ restricted to $\Omega'$.  $u_0$ and $u'$ are converted into vectors using the natural ordering $(i+1,j) > (i,j)$ then $(i,j+1) > (i,j)$.

\subsection{The finite difference operators $D_1$, $D_2$, and $\Delta$}

Let $D_1$, $D_2$, and $\Delta$ be matrices defined on the interior of $R$ so that
\begin{eqnarray*}
D_1 u_{i,j} = \frac{1}{2}(u_{i+1,j} - u_{i-1,j}),\\
D_2 u_{i,j} = \frac{1}{2}(u_{i,j+1} - u_{i,j-1}), \text{ and }\\
\Delta u_{i,j} = (u_{i+1,j} + u_{i-1,j} + u_{i,j+1} + u_{i,j-1} - 4u_{i,j})
\end{eqnarray*}
for all $(i,j) \in \Omega$.  We use second order centered differences to approximate the first partial derivatives and the five point discretization of the Laplacian. Define also $D_1 \Delta = D_1 * \Delta$ and $D_2 \Delta = D_2*\Delta$.  In order to enforce the boundary conditions, we eliminate the rows of each of these operators corresponding to pairs $(i,j) \in R - \Omega$ and columns corresponding to pairs $(i,j) \in R - \Omega'$.  Finally we define $(I - \Delta)^{-1}: \Omega \rightarrow \Omega$ to be the resolvent of the Laplacian.

\subsection{Discretizing the gradient}

Here we discretize the Sobolev gradient.  Note that in the finite dimensional case, equation \eqref{gradE1} holds true.  First we discretize $F(Du)$ as
\begin{equation} \label{fdisc}
F(Du) = D_2 u' \ast D_1 \Delta u' - D_1 u' \ast D_2 \Delta u'
\end{equation}
where for two vectors $x$ and $y$, $x \ast y$ denotes the vector consisting of their element wise product.  Next we discretize the Euler-Lagrange equation as follows.  Discretize $\vec{v}$ given in \eqref{eulerlagrange} so that
\begin{eqnarray*}
v_{\alpha_1}= - F(Du) \ast \binom{ -D_2 \Delta u'}{D_1 \Delta u'}\\
v_{\alpha_2}=   F(Du) \ast \binom{ -D_2 u'} {D_1 u'}.
\end{eqnarray*}
Then the discretized Euler-Lagrange equation for this problem becomes
\begin{equation}\label{eldisc}
g_{EL} = D_1' v_{\alpha_1}(1) + D_2' v_{\alpha_1}(2) + \Delta( D_1' v_{\alpha_2}(1) + D_2' v_{\alpha_2}(2))
\end{equation}
where $D_i'$ denotes the adjoint of $D_i$.

\subsection{Smoothing the Euler-Lagrange equation}

In order to use the Sobolev gradient as given in equation \eqref{gradE1}, we need to discretize the operator $(I+D^*D)$.  Once we have a discretization of this operator, we obtain the discrete Sobolev gradient by solving the equation
\begin{equation*}
(I+D^*D) g_{Sob} = g_{EL}
\end{equation*}
Since $D$ involves partial derivatives or order up to three, a discretization of this operator would force us to solve a very complex linear system which can be computationally very time consuming.  A better approach is to obtain an equivalent operator to use for smoothing the Euler-Lagrange equation.  A good choice turns out to be $(I - \Delta)^3$.  A good explanation for this is given in \cite{calder2011}.  Thus in order to solve for the Sobolev gradient, we solve the linear system
\begin{equation}\label{sobgrddsc}
(I-\Delta)^3 g_{Sob} = g_{EL}.
\end{equation}
This is computationally the most expensive part of the algorithm.  To solve the linear system, a Cholesky factorization of $(I-\Delta)$ is performed to obtain an upper and lower triangular matrix.  Then the linear system
\begin{equation*}
(I-\Delta) x = y.
\end{equation*}
is solved three times using the factorization in order to apply the smoothing operator three times. This is convenient method since  homogeneous Dirichlet boundary conditions naturally apply at each step of the factorization.

\subsection{Discretization in time}

In order to discretize the system \eqref{flow} in time, we use an explicit Euler scheme with a locally minimizing time step.  Thus we generate a sequence of images according to
\begin{equation*}
u_0(n+1) = u_0(n) - t_n g_{Sob}(n).
\end{equation*}
We remark that several options are available for time discretization.  One can choose between an explicit or implicit scheme.  One can also use a fixed time step.  For our purpose, we chose the explicit scheme since computing the gradient here is more expensive than computing the value of the functional.  Thus an implicit scheme that would require us to compute the gradient several times at each step is a lot more expensive than an explicit scheme with a line search algorithm to compute locally the optimum time step.

%%%%%%%%%%%%%%%%%%%%%%%%%%%%%%%%%%%%%%%%%%%%%%%%%%%%%%%%%%%%%%%%%%%%%%%%%%%%%%%%%%%%%%%
\section{Results} \label{secresults}

We present here the results from our simulations.  The implementation of the discrete operators derived in the previous section using finite differences is quite straightforward if a language with vector programming capabilities is used. The algorithm is programmed in Matlab and all operations are written in  matrix form. We run our simulations on an IBM laptop with 1.5 GB RAM and 1.4GHz Pentium M processor. In each case the starting point for the simulation was the solution of Laplace's equation
\begin{equation}
\Delta u = 0 \text{ and } u|_{\partial \Omega} = g
\label{eq-laplace}
\end{equation}
which was solved by using successive over relaxation method.  The boundary condition function $g$ takes values from the neighboring non-inpainted domain. In the examples we present, we normalize all images so that the maximum pixel value is one. This choice for the initial guess of the solution looks natural if the Navier-Stokes analogy is used: equation \eqref{eq-laplace} then represents an irrotational fluid field, \ie with zero vorticity $\omega = \Delta \psi = 0$, that will be iterated to find a nonzero vorticity (source) term that is present in the model. In the method of BBS, the initial condition is set after performing several steps of anisotropic diffusion following \eqref{anis}.

For comparison purposes, we also implemented the method of BBS  as is described in \cite{bertalmio2000}.  We incorporate slope limiters for convection term to propagate fronts with sharp variation and apply every 50 time steps a few iterations of the Perona-Malik scheme  for anisotropic diffusion to avoid spurious  oscillations:
\begin{equation}
I'(t) = \nabla \cdot c (x , y ,t) \nabla I(t) \text{ where } c(x,y,t) = e^{ - \|\nabla I\|/k}.
\end{equation}
As many of the specific details such as the term used for anisotropic diffusion are not available in the original paper of BBS, an exact reproduction of their results proved to be difficult. As a consequence, the comparison presented in the following should be regarded with caution. Nevertheless, reported convergence properties of the method of BBS are recovered by our implementation.

\subsection{Comparison with the original implementation of the Navier-Stokes model by  Bertalm{\'\i}o, Bertozi and Sapiro}

We perform a full quantitative comparison displaying convergence curves using our minimization procedure and the Navier-Stokes model of BBS.  In the original model it is proposed to solve the equation $\nabla^{\bot} u \cdot \nabla \Delta u = 0$ by evolving this PDE in time adding anisotropic diffusion.  Typical methods coming from Fluid dynamics literature are used, and therefore, 
several delicate and technical choices are necessary: the choice of the limiters for the discretization of convection terms, the relaxed Laplacian solvers for the vorticity needing the tuning of a constant, the choice of how often and under which form to introduce the diffusion term to avoid spurious oscillations. Besides, no convergence analysis is provided to state that a solution of the PDE is obtained within a reasonable number of iterations. 
A rough estimation of the necessary steps for convergence is given in \cite{bornemann}, using the CFL time-step restriction for the explicit scheme used by BBS. The number of steps to reach a stationary state is estimated as large as $N_{pixels}^{1/2}$, which results in several thousands of iterations for typical test cases. Our implementation of the original method displays convergence properties of this order of magnitude.

Our method does not display these shortcomings related to the practical implementation: there are not constants to be tuned and
no advanced knowledge of numerical methods for Navier-Stokes equations is required. Since we use a minimization technique, we can choose a locally minimizing time step to ensure that the energy is minimized.  In addition, by preconditioning the Euler-Lagrange equations, we remove the need to apply anisotropic diffusion.  In contrast with the method of BBS, the results for inpainting with and without anisotropic diffusion were nearly identical for the test cases we tried.  

We first evaluate convergence properties of our method, compared to the method of BBS, for the test case displayed in Fig. \ref{grsm}.  Two minimization schemes were used, based upon the Euler-Lagrange equations and the Sobolev gradient using just one application of the preconditioner (also refered to as the $H^3$ gradient), respectively.  For the purpose of demonstrating the effect of preconditioning on image inpainting, we found it sufficient and reasonable to apply the preconditioner $(I - \Delta)^{-1}$ just once.  Thus in the later images in order to save CPU time, we used just one application of preconditioning.  

We plot in Fig. \ref{quant} the reduction in $\| \nabla^{\bot} u \cdot \nabla \Delta u \|^2$ using the method of BBS and our energy minimization method.  We can observe that our method is quantitatively closer to the solution of the PDE.  With our minimization scheme we were able to achieve a decrease of two or three orders of magnitude for most cases.  As our method is a gradient descent method, we expect a linear rate of convergence.  However, we do note that by using the Sobolev gradient in place of the Euler-Lagrange equations, the convergence to a zero is slightly accelerated.

In the same figure \ref{quant} we plot the "error"  $\varepsilon= \max|u_{n+1} - u_n|$, where $u_n$ represents the current image. This quantity is used as stoping criterion for the iterative process since it seems to be the best measure for convergence in image processing, as it is generally a valid measure of convergence in numerical analysis. We should note in passing that, compared to Fluid dynamics applications where $\varepsilon < 10^{-8}$ is required for convergence, for image processing applications no apparent change in the image is observed starting with $\varepsilon < 10^{-4}$.
\begin{figure}[h!]
\centering
\includegraphics[width=.45\columnwidth]{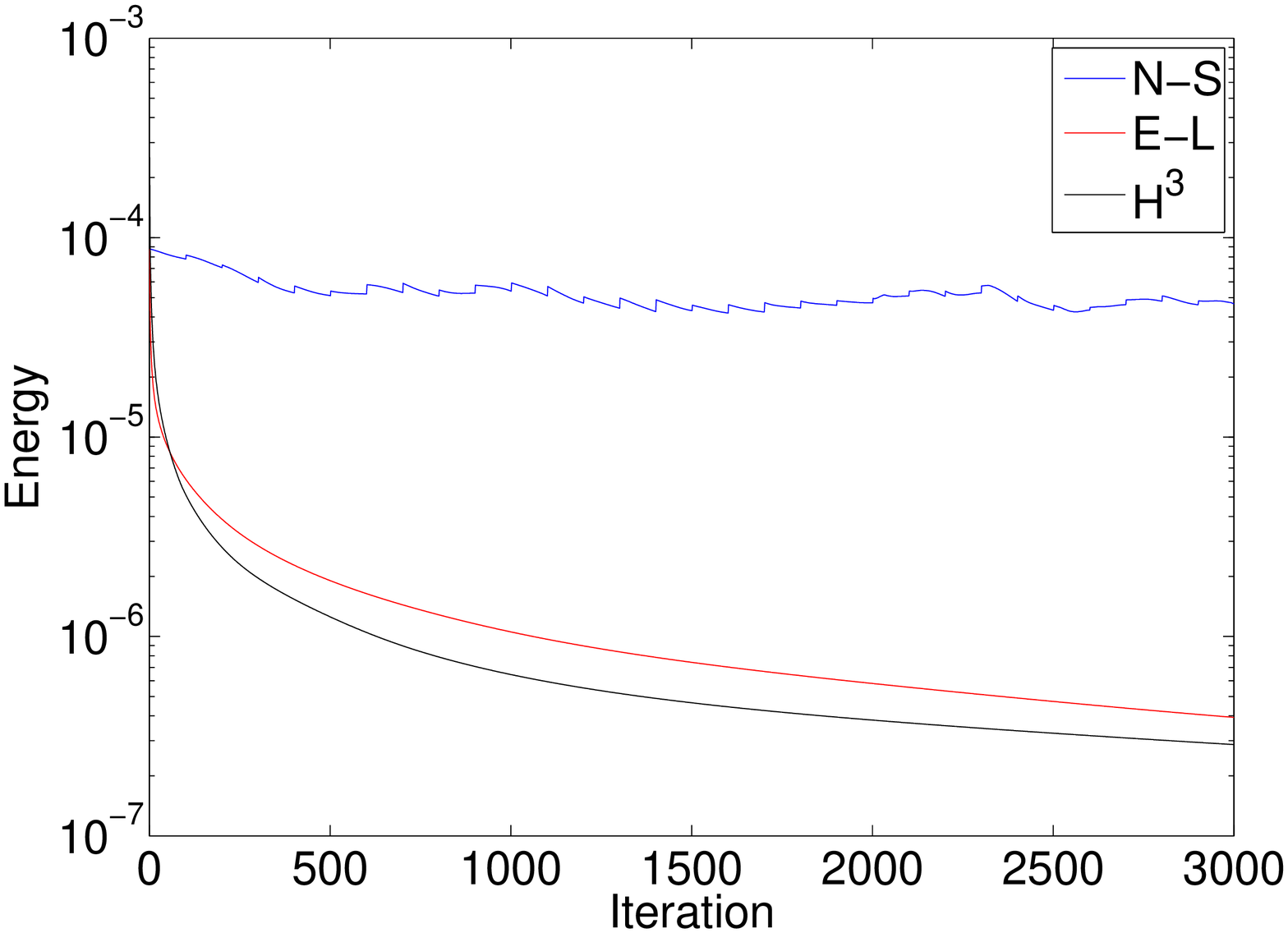}
\includegraphics[width=.45\columnwidth]{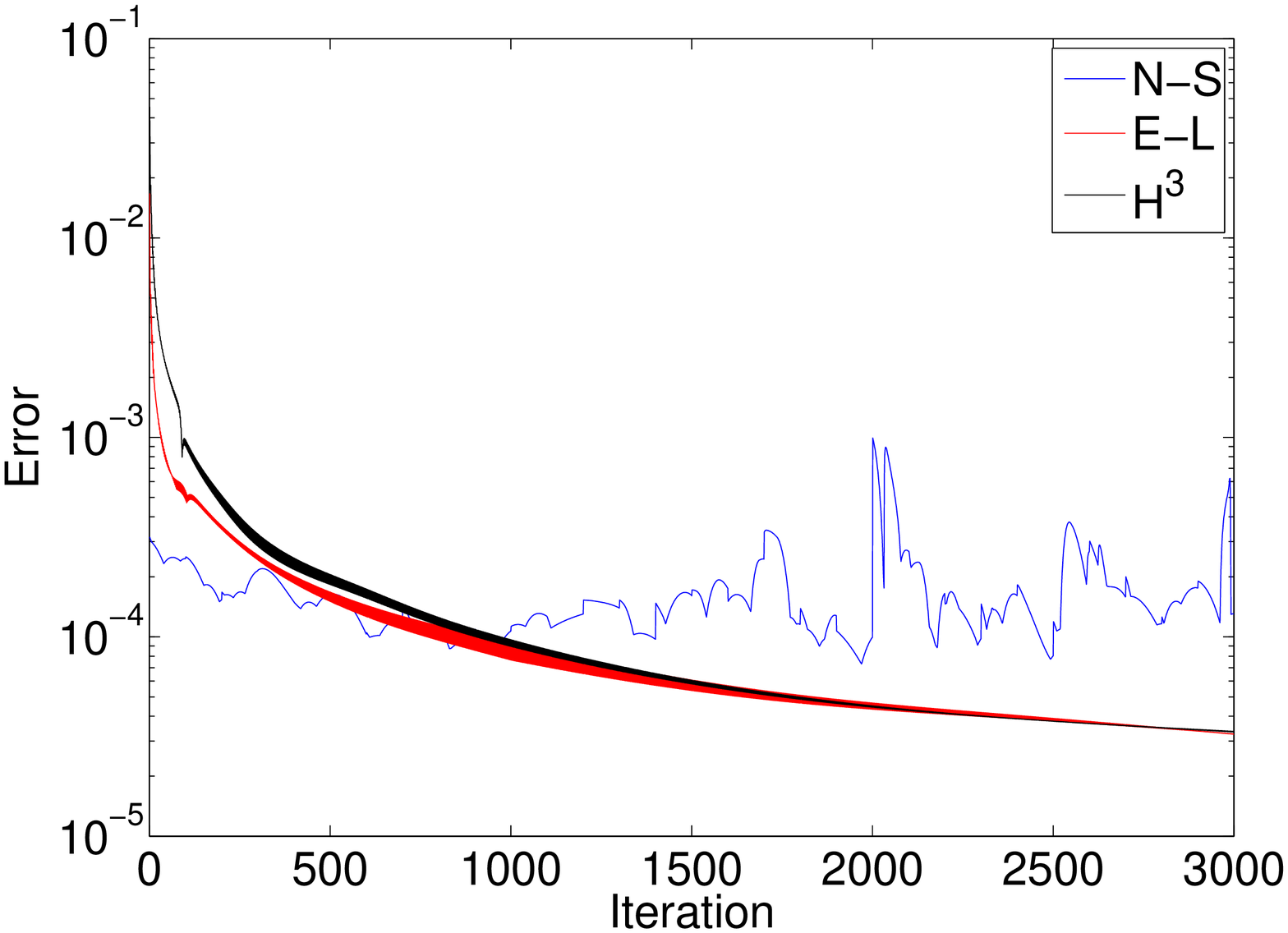}
\caption{Left: Comparison of the norm of the residual (energy)  for the different methods.  Right: The error as defined by $max \ |u_{n+1} - u_n |$ for the different methods.  N-S denotes the Navier-Stokes method following the implementation of Bertalm{\'\i}o \etal  E-L denotes the iterative method using the Euler-Lagrange equations.  $H^1$ denotes the Sobolev gradient.}\label{quant}
\end{figure}

An important question related to convergence is the condition number of the algorithm.  Although the condition number of the least squares problem when one considers minimization with the Lagrange equations is relatively high, we demonstrate, that this number is significantly reduced when we precondition the problem and take the gradient in a higher order Sobolev space.  Since these results concern more general numerical analysis considerations, we report them in the Appendix  \ref{conditioning}.

The images in Fig. \ref{grsm} were ran until an error tolerance of $10^{-4}$ was achieved.  Using the method of BBS, this took 9282 iterations in a CPU time of 67.28 seconds.  Using energy minimization with the Euler-Lagrange equations, this took 1014 steps in a CPU time of 29 seconds.  Using one application of the preconditioner, convergence was reached in 742 iterations in a CPU time of 32 seconds.  Using three application, convergence was reach in 942 step in about 40 seconds.  We remark that by preconditioning we were able to reach the error tolerance using a lesser number of iterations.  The CPU time was comparable for the case using the Euler-Lagrange equations and the Sobolev gradient gradient, as solving a linear system at each time step adds to the processing time.  However, one could take further measures to decrease this additional time by implementing a different algorithm for solving the linear system.

Qualitatively the effect of preconditioning is apparent in the bottom two images in figure \ref{grsm}.  While the image resulting from the Euler-Lagrange equations has some noise where the inpainting was performed, the image resulting from the Sobolev gradient is smooth in the inpainting region.  Also note that the edge between the face and the background is correctly placed.  In this sense, we remove the need to apply anisotropic diffusion.
\begin{figure}[h!]
\centering
\includegraphics[width=.45\columnwidth]{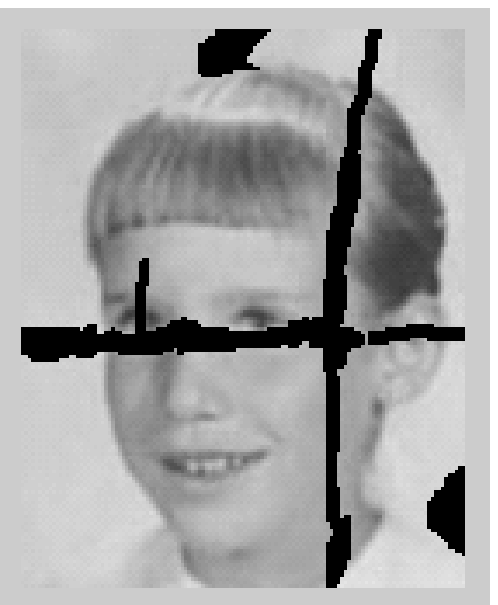}
\includegraphics[width=.45\columnwidth]{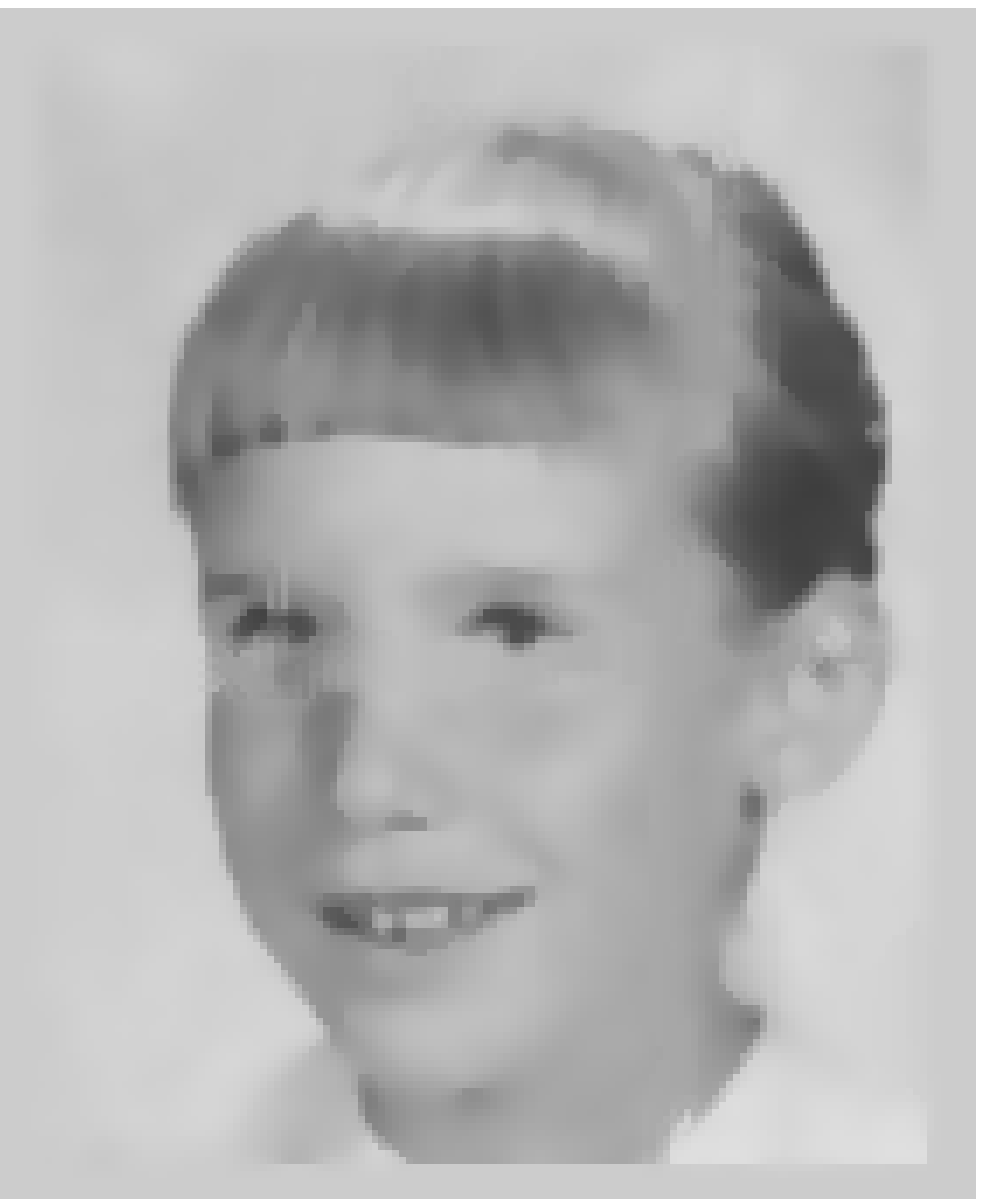}\\
\includegraphics[width=.45\columnwidth]{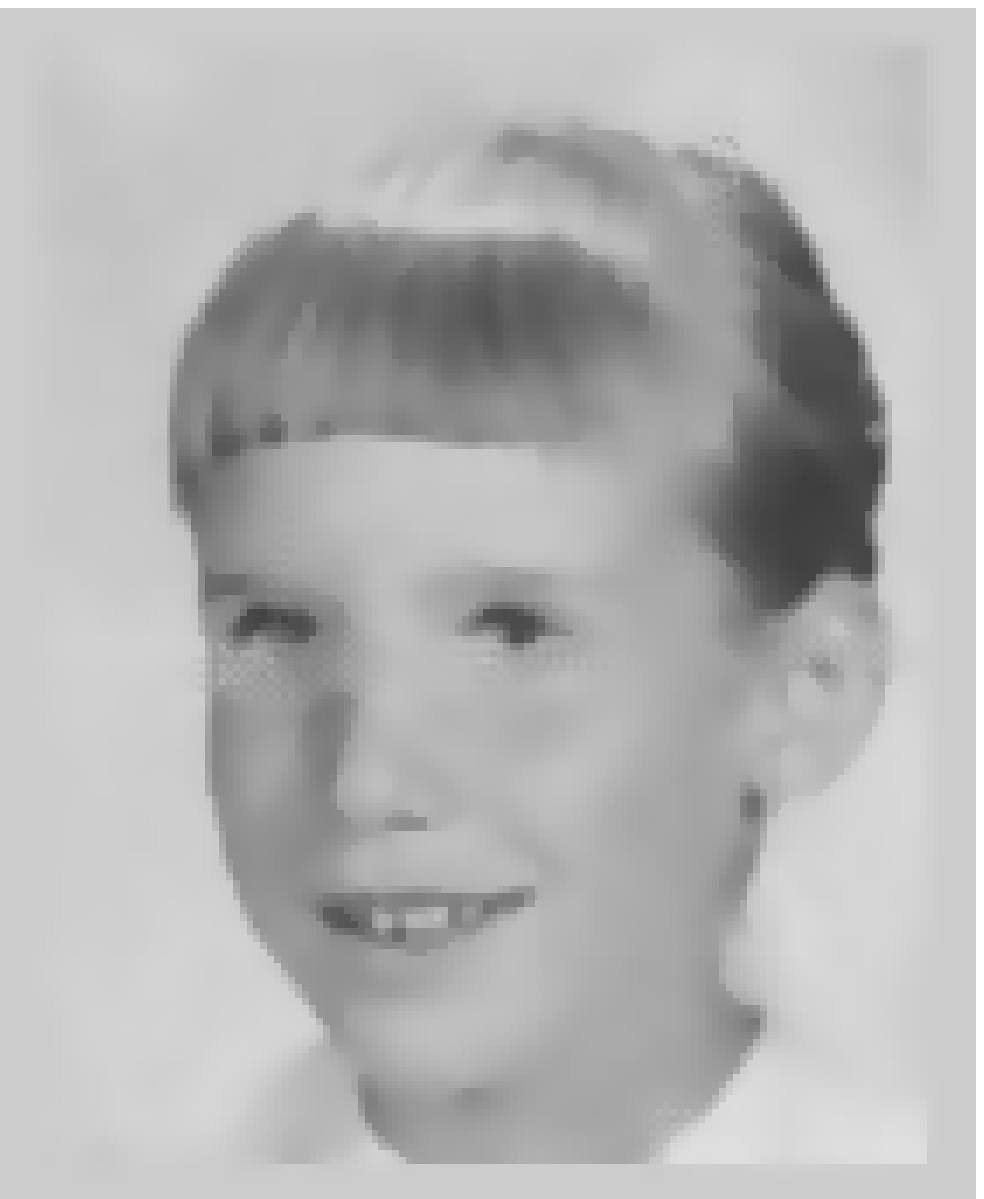}
\includegraphics[width=.45\columnwidth]{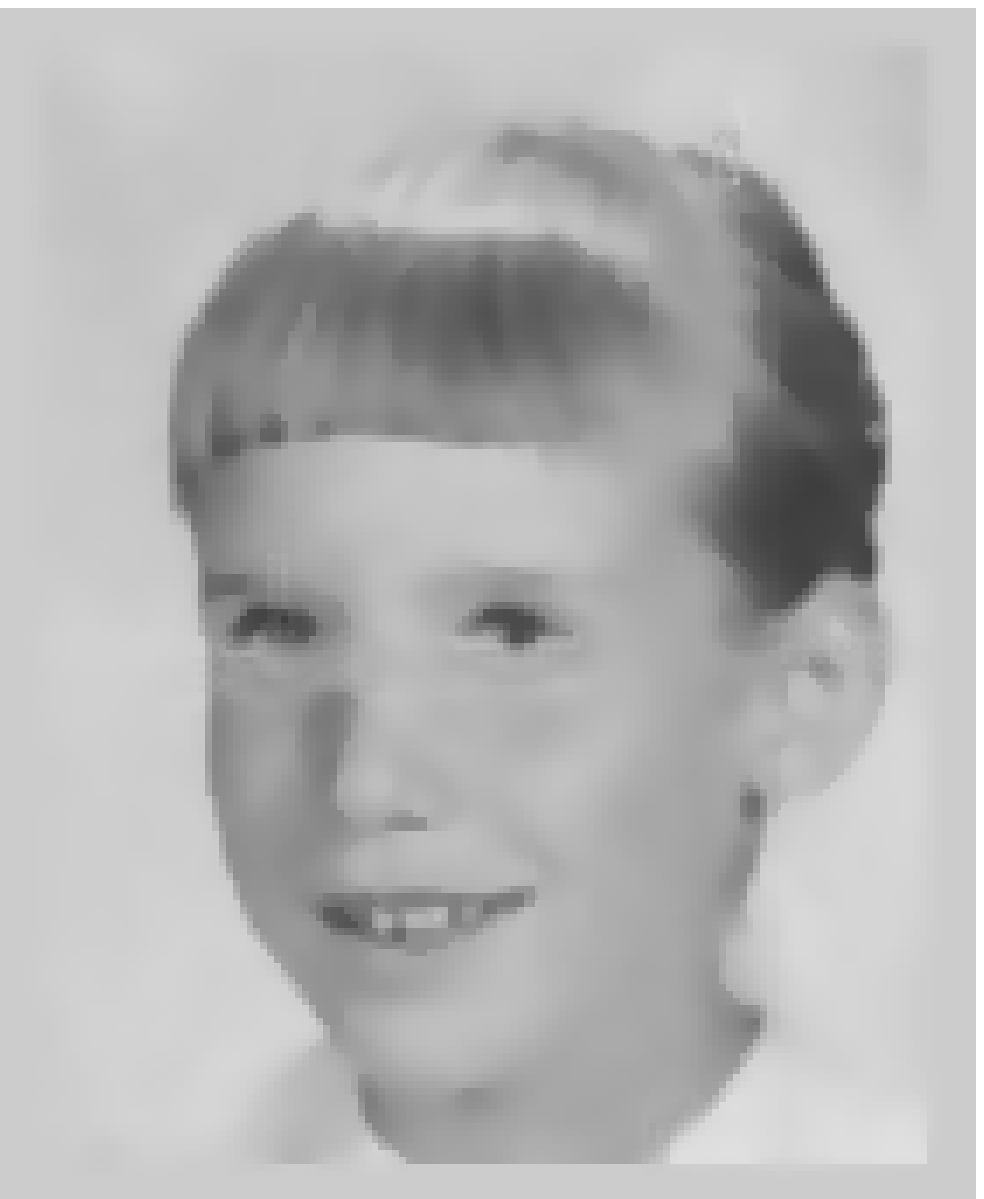}
\caption{The resulting images for the simulations.  Top left: the defective image, Top right: Navier-Stokes inpainting, Bottom left: Energy minimization with the Euler-Lagrange equations, Bottom right: Energy minimization with the $H^1$ gradient.}\label{grsm}
\end{figure}

In this image the result we present for the BBS method is qualitatively worst than the result presented in the original paper of Bertalm{\'\i}o \etal \cite{bertalmio2000}.  We attribute this to the fluctuating error as presented in figure \ref{quant}; the result for this problem initially improved, however the subsequent images  did not correctly inpaint the eye.  This behavior is certainly related to the careful choice of the parameters in the original implementation (time step, when and how to apply the anisotropic diffusion). Since the error analysis is not available for the method of BBS, we didn't attempt to find a deterministic choice for these parameters.

Since the full image of girls, first appeared in \cite{bertalmio2000}, is currently used to test inpainting methods, we present in  Fig. \ref{girls} our results for this test case. The inpainting region is correctly filled, with a better image quality obtained when the Sobolev gradient method is used.  
\begin{figure}[!h]
\includegraphics[width=.45\columnwidth]{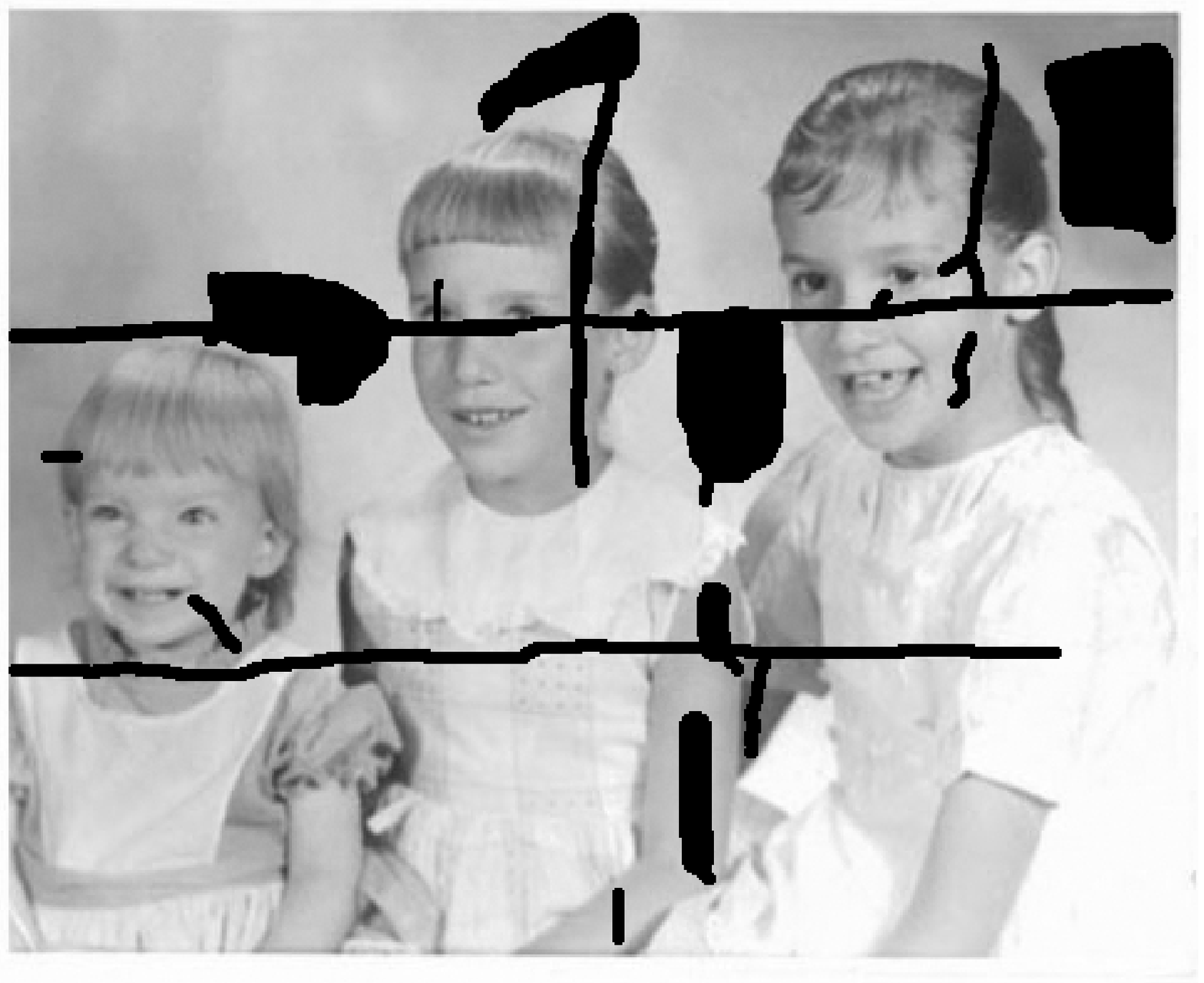}
\includegraphics[width=.45\columnwidth]{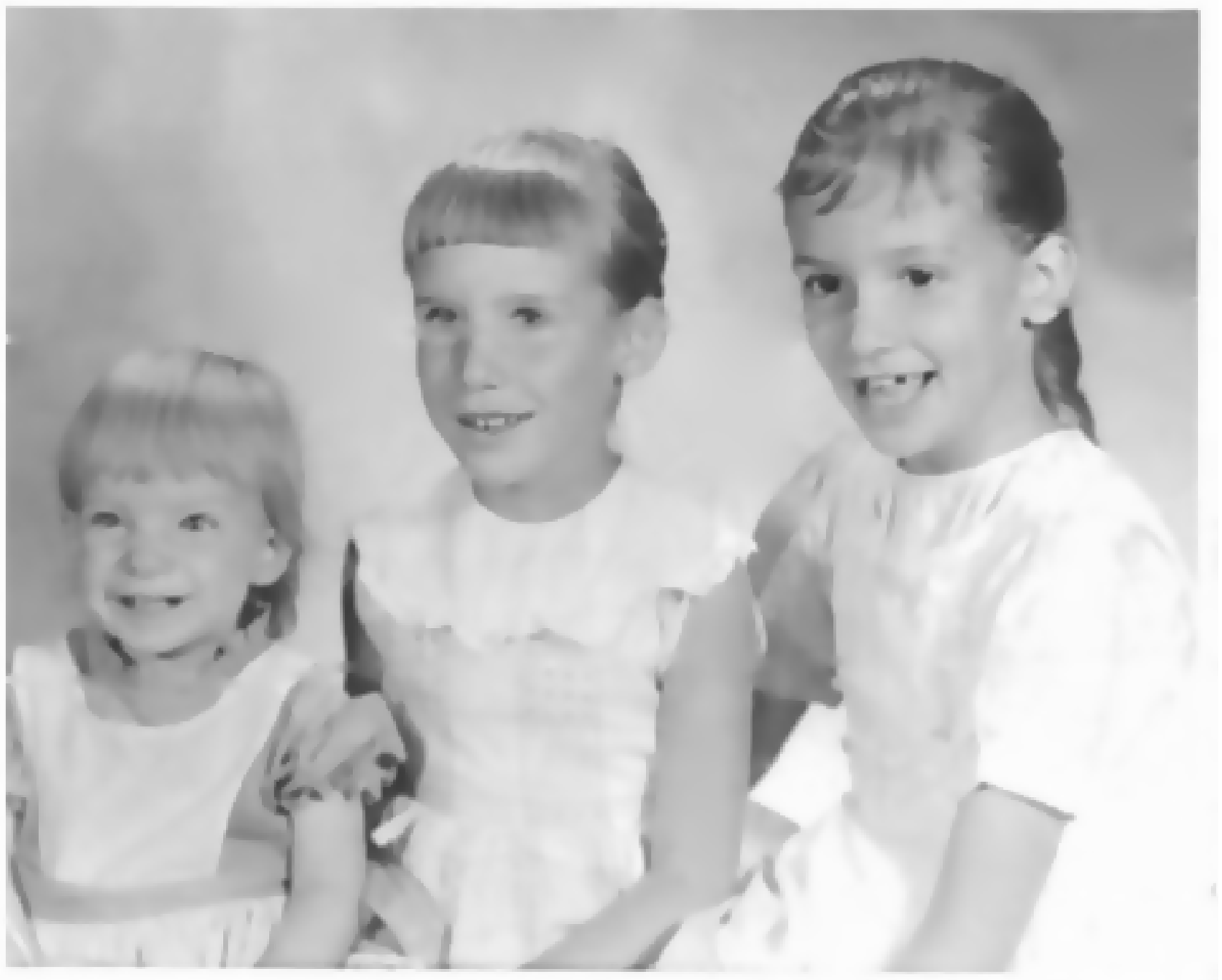}\\
\includegraphics[width=.45\columnwidth]{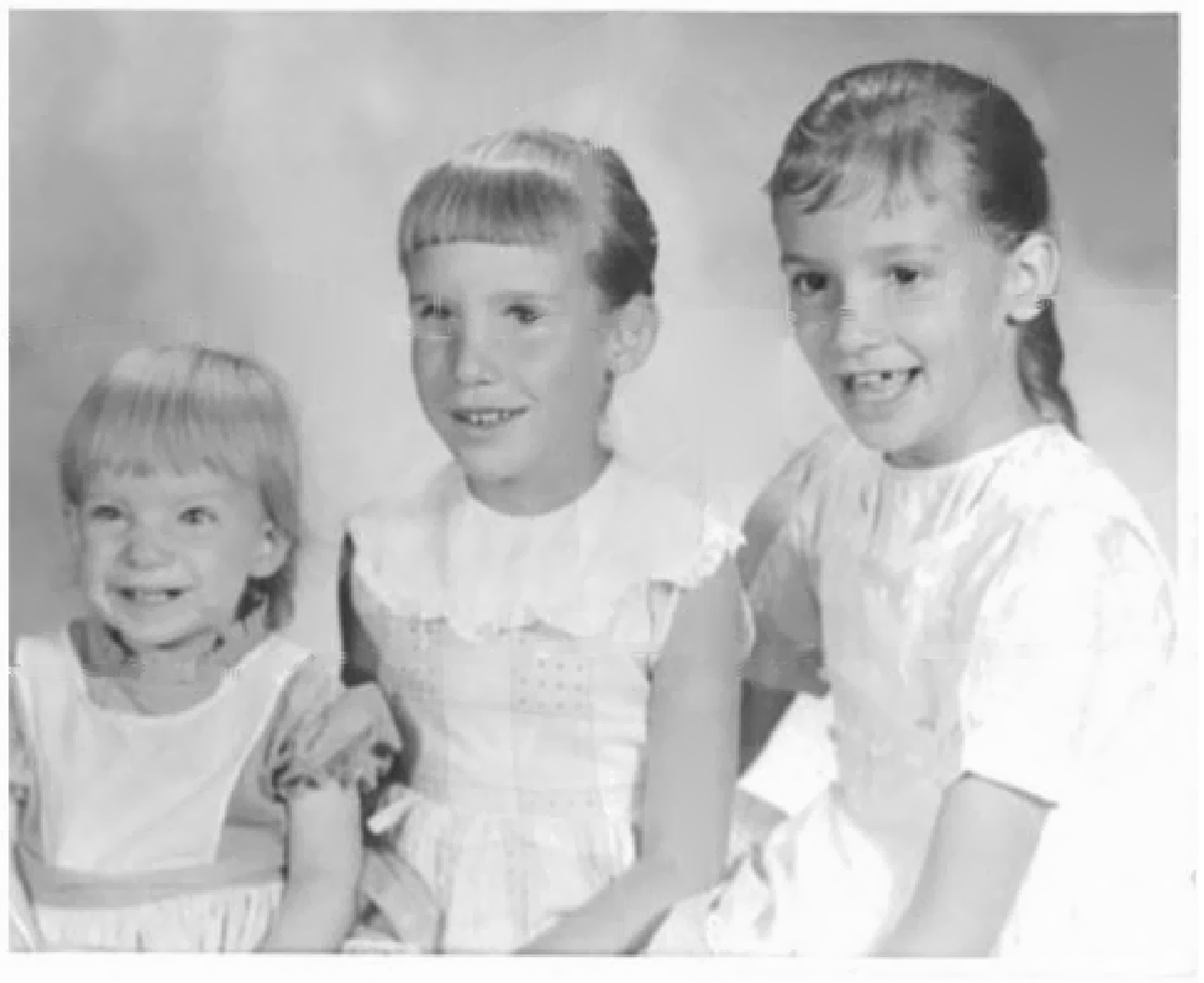}
\includegraphics[width=.45\columnwidth]{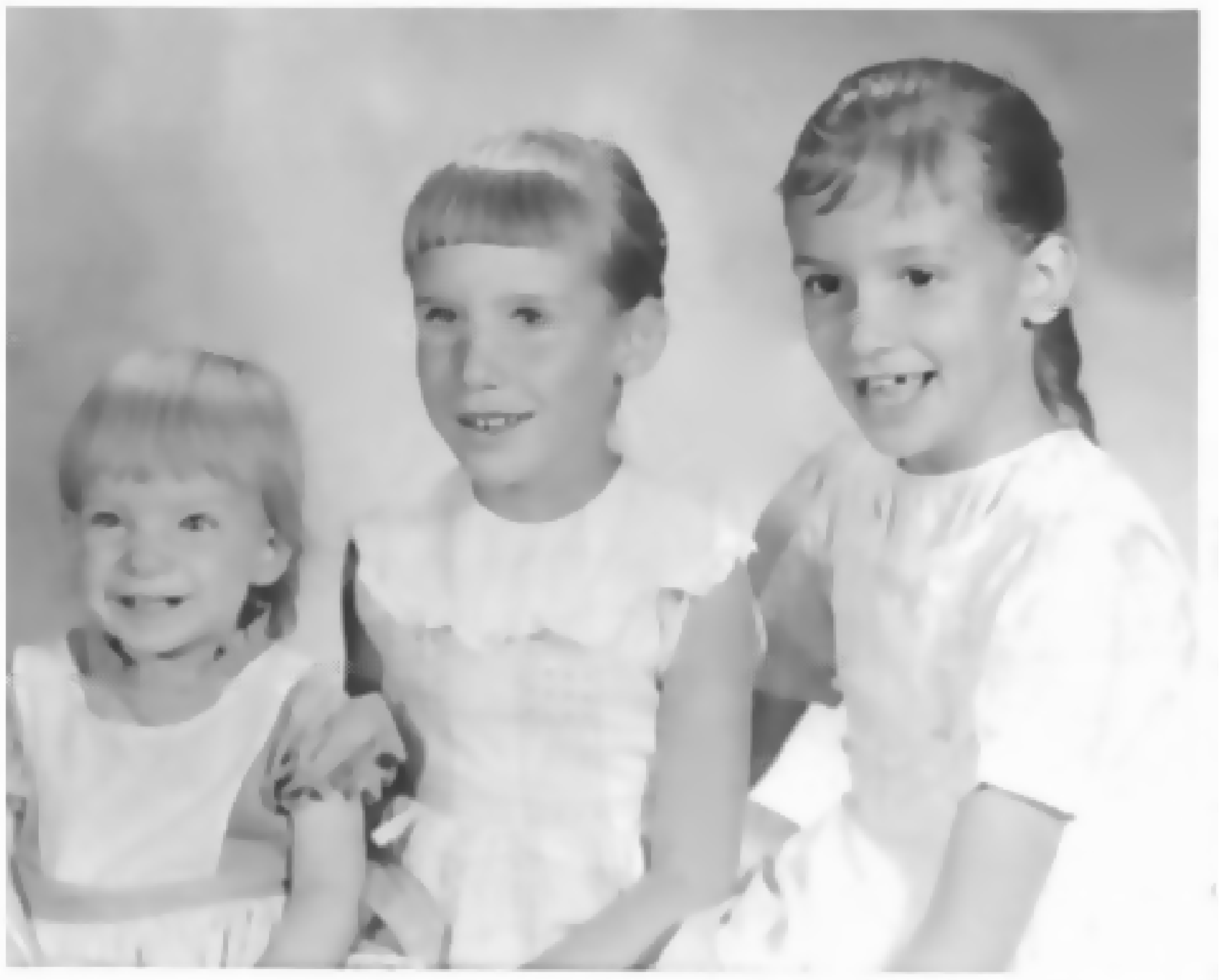}
\caption{Inpainting of the girls.  Compare with the original papers \cite{bertalmio2000} and \cite{bornemann}.  Top left: original image, top right: the solution of Laplace equation, bottom left: the solution using Navier-Stokes inpainting, bottom right: energy minimization with the $H^1$ gradient.}\label{girls}
\end{figure}

\subsection{Qualitative comparison with other methods}

In an effort to compare our results to other methods used in image inpainting, we use here the image \ref{eye} taken from \cite{bornemann}.
We performed our inpainting algorithm just on the gray scale image as the issues of edge recovery and smoothness of the interpolation can be determined from this image.

In comparison with the method of Bornemann and M{\"a}rz \cite{bornemann}, the construction of the iris is comparable.  In their result a better quality image is achieved as the texture of the eyebrow is interpolated.  Also the CPU for this example was about 60 seconds.  In \cite{bornemann} a reported CPU time of 0.4 seconds was reported. It goes without saying that our iterative method could not compete in terms of efficiency with a non-iterative fast-marching method, such that proposed in \cite{bornemann}. Nevertheless, the method of Bornemann and M{\"a}rtz needs a careful tuning of four parameters, which is not the case for our method (see also \cite{survey}).

\begin{figure}[!h]
\includegraphics[width=.45\columnwidth]{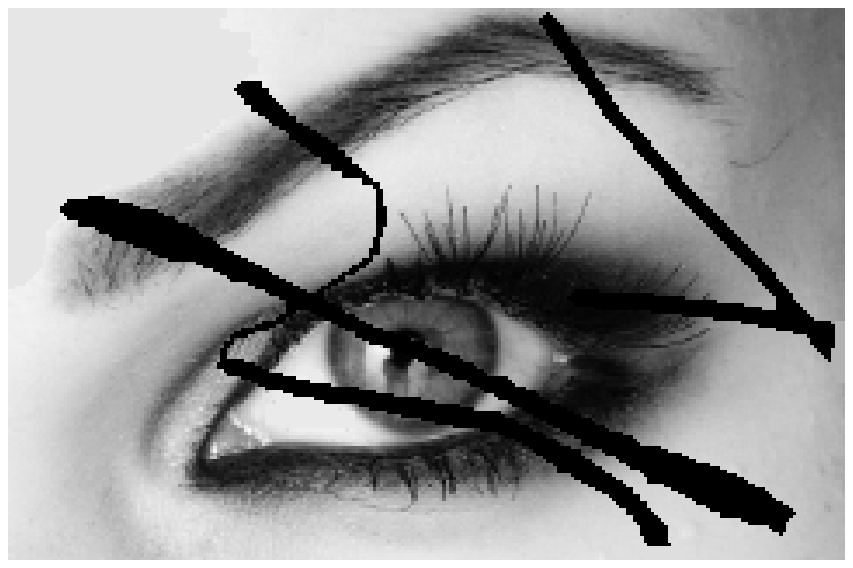}
\includegraphics[width=.45\columnwidth]{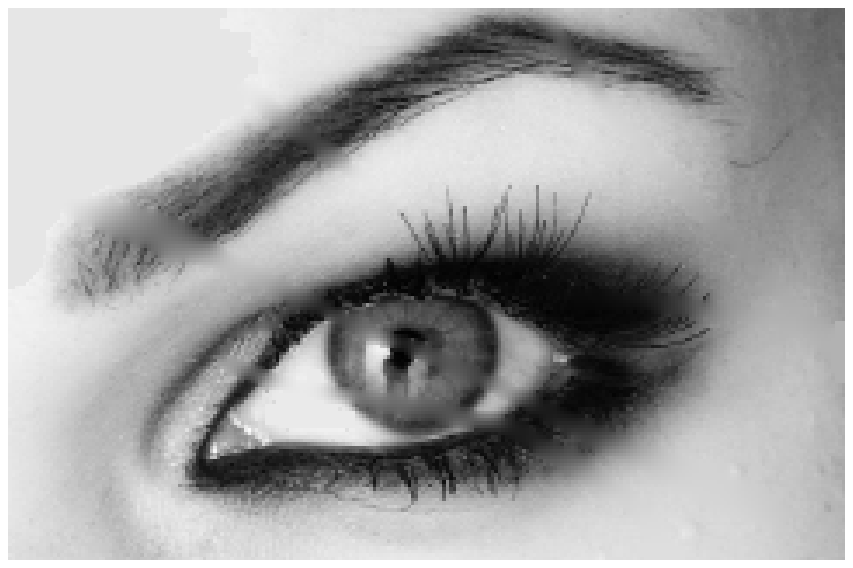}\\
\includegraphics[width=.45\columnwidth]{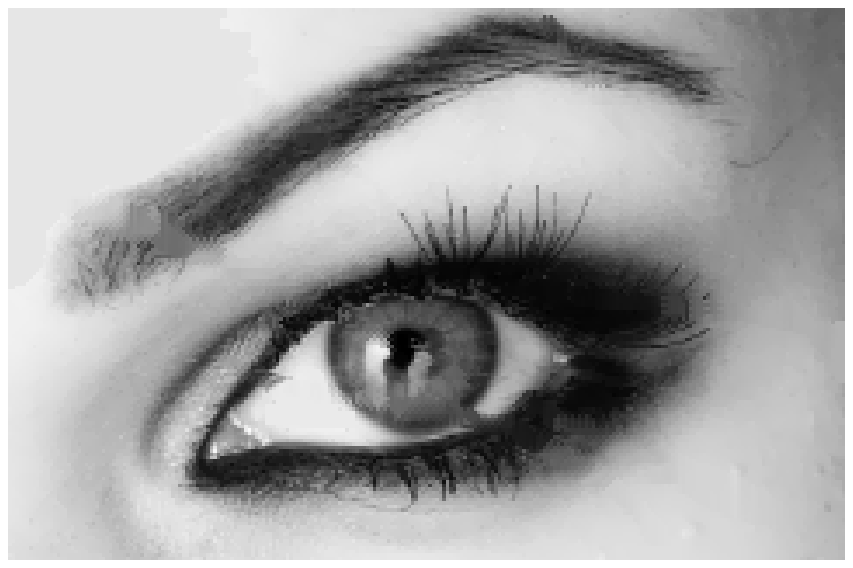}
\includegraphics[width=.45\columnwidth]{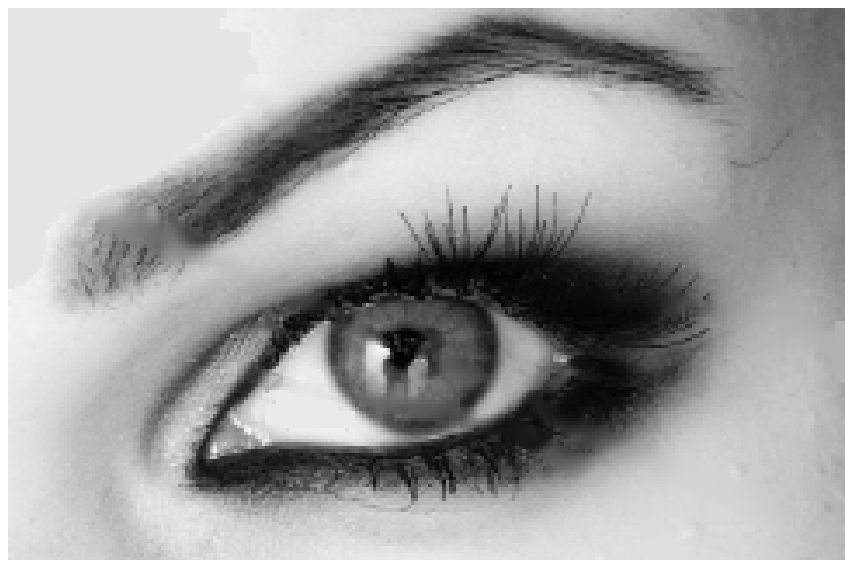}
\caption{Inpainting of the eye as was done in \cite{bornemann}.  Top left: original image, top right: the solution to Laplace's equation, bottom left: the solution using Navier-Stokes inpainting, bottom right: energy minimization with the $H^1$ gradient.}\label{eye}
\end{figure}
\subsection{Image interpolation}

In this section, we apply our algorithm to perform image interpolation.  We start with the two images given in figure \ref{interp} and interpolate the bird to 160 by 152 pixels and Lena to 200 by 256 pixels using nearest neighbor interpolation in figure \ref{interp}. The interpolation is carried out as follows.  The image is first expanded to the dimensions of the larger image where the newly added pixel designate the inpainting region.  As with inpainting, we first solve Laplace's equation then apply the inpainting algorithm to further sharpen the image.  
\begin{figure}[h!]
\centering
\includegraphics[width=.45\columnwidth]{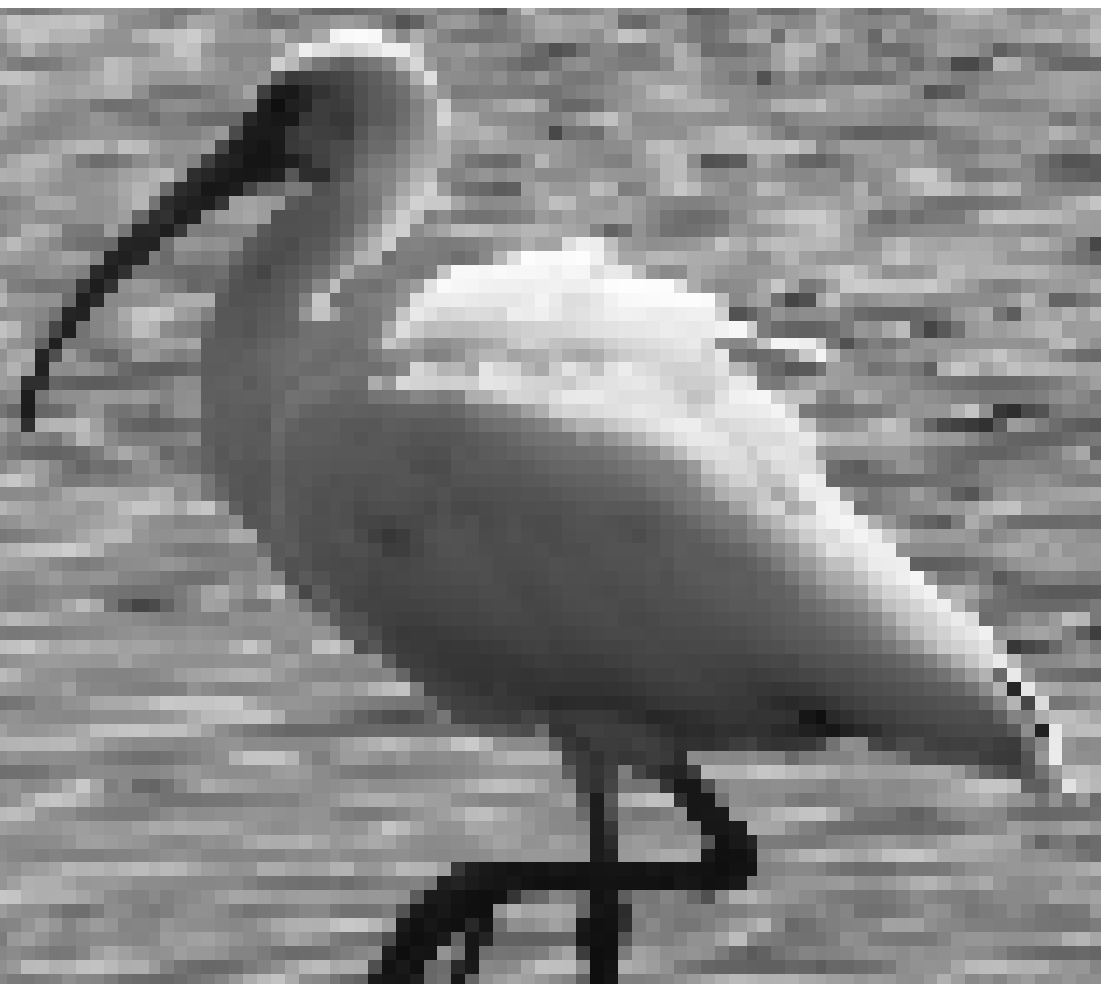}
\includegraphics[width=.45\columnwidth]{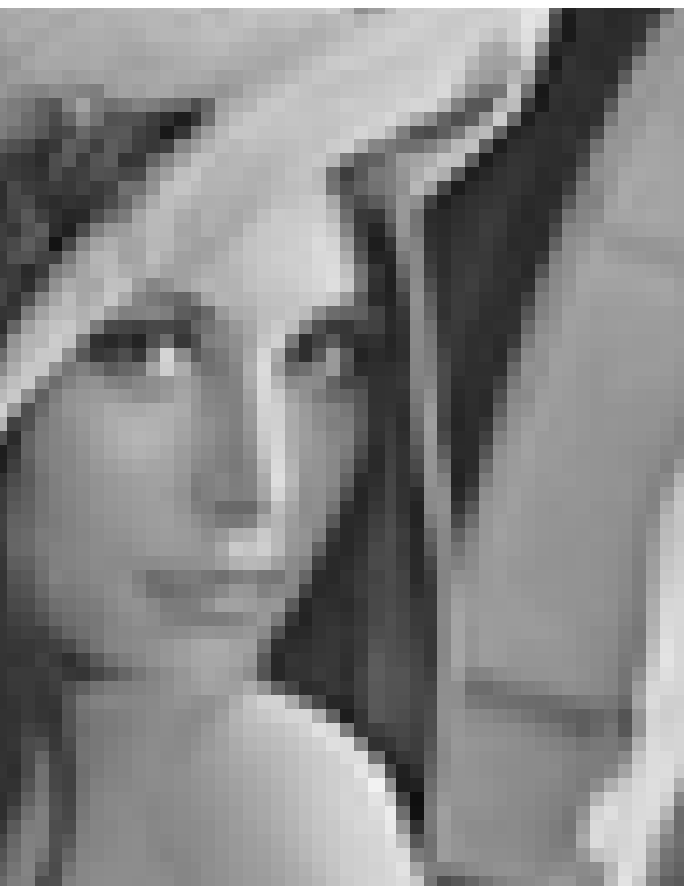}\\
\caption{ Image interpolation: original image enlarged using nearest neighbor interpolation.  Left: bird original size 80 by 76 pixels, Right: Lena original image 50 by 64 pixels.} \label{interp}
\end{figure}

For these two images the effect of using anisotropic diffusion in the method of BBS is evident.  The smoothing effect of anisotropic diffusion does not allow one to recover the detail of these two images.  In the energy minimization method, since we avoid using anisotropic diffusion, we can recover sharp edges as well as some of the finer detail.  For example in the image of the bird, notice that in the energy minimization method using both the Euler-Lagrange equation and the Sobolev gradient, the eye of the bird was actually recovered.  The effect of preconditioning here is evident in that it minimizes some of the artifacts associated with image interpolation.  For example comparing the beak of the bird in the image resulting from cubic interpolation and minimization with the Lagrange equation vs minimization with the Sobolev gradient, we see that the latter results in sharper edges. Similar effects are evident when we examine the interpolation of Lena.  In particular, the smoothing and sharper edges are most visible in the upper part of the hat of Lena.

\begin{figure}[h!]
\centering
\includegraphics[width=.45\columnwidth]{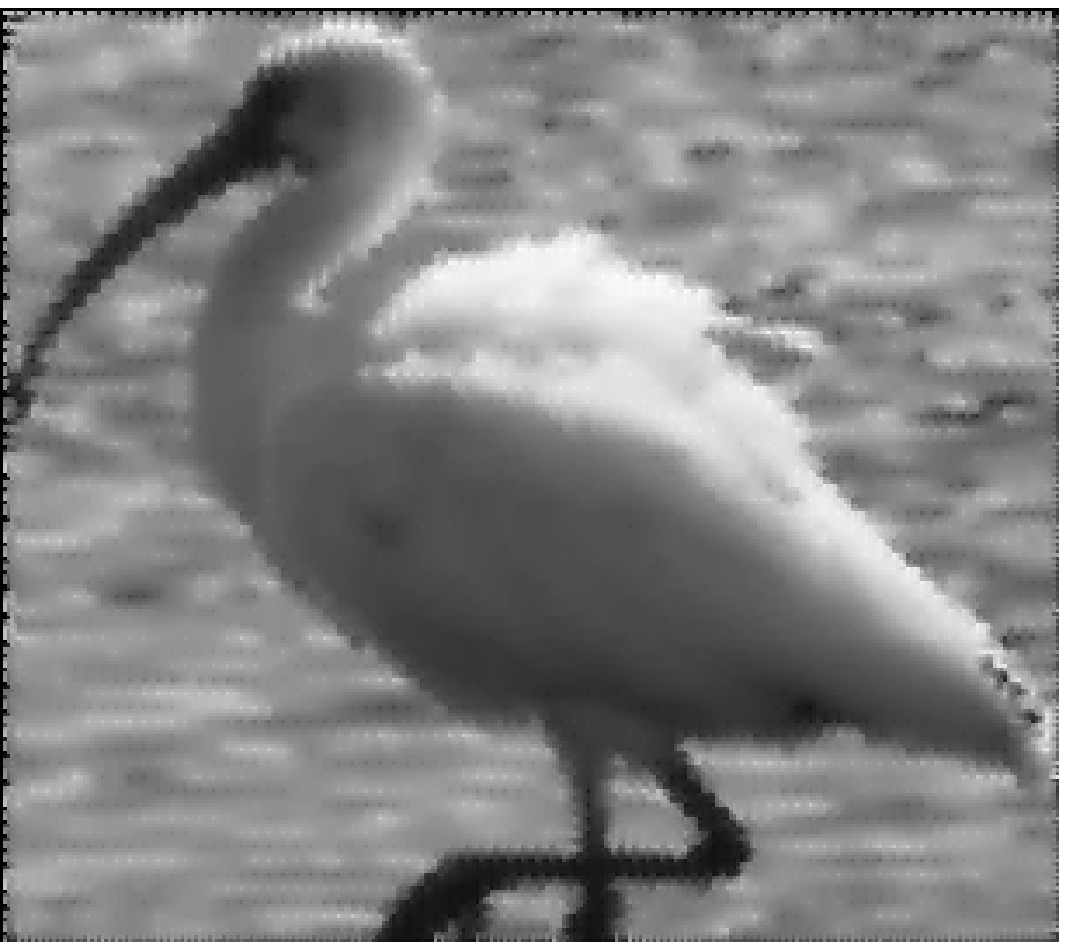}
\includegraphics[width=.45\columnwidth]{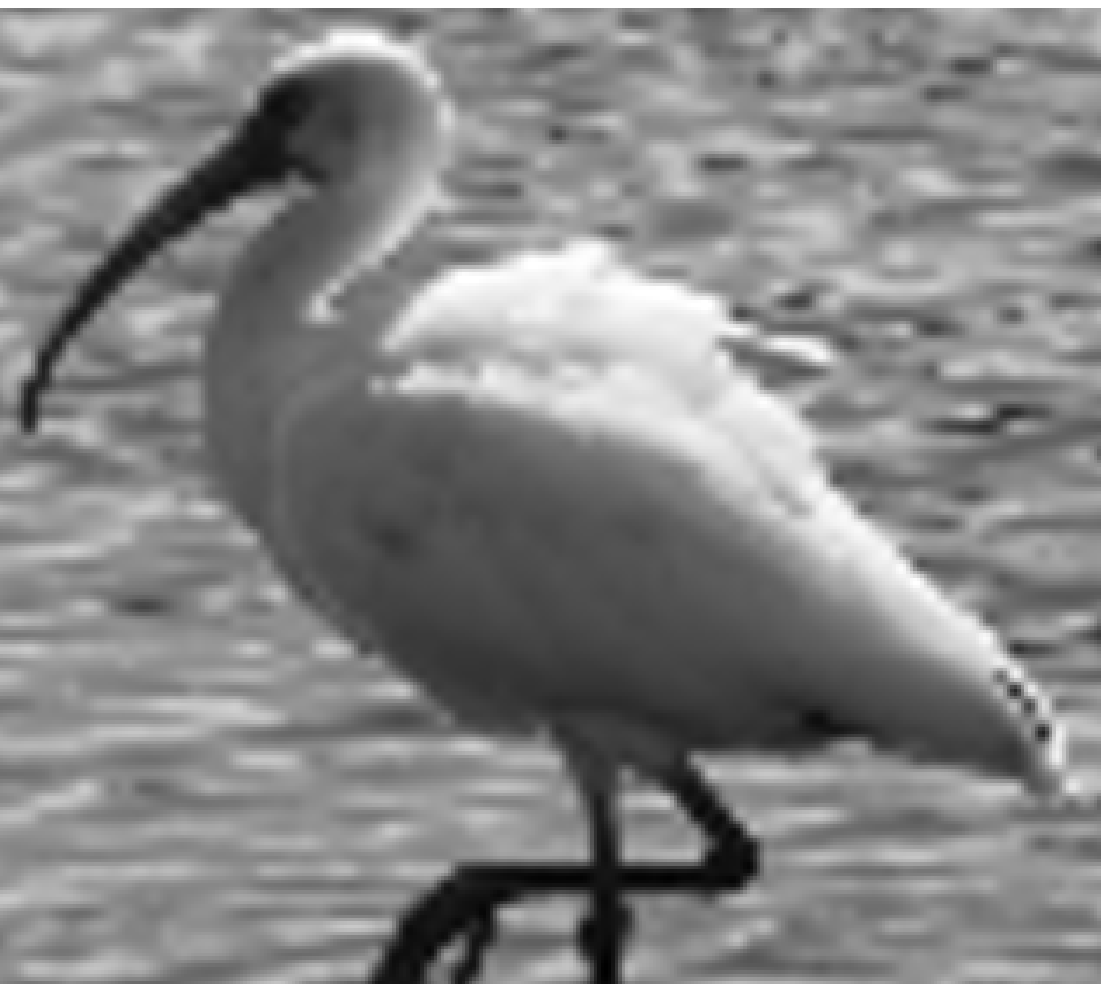}\\
\includegraphics[width=.45\columnwidth]{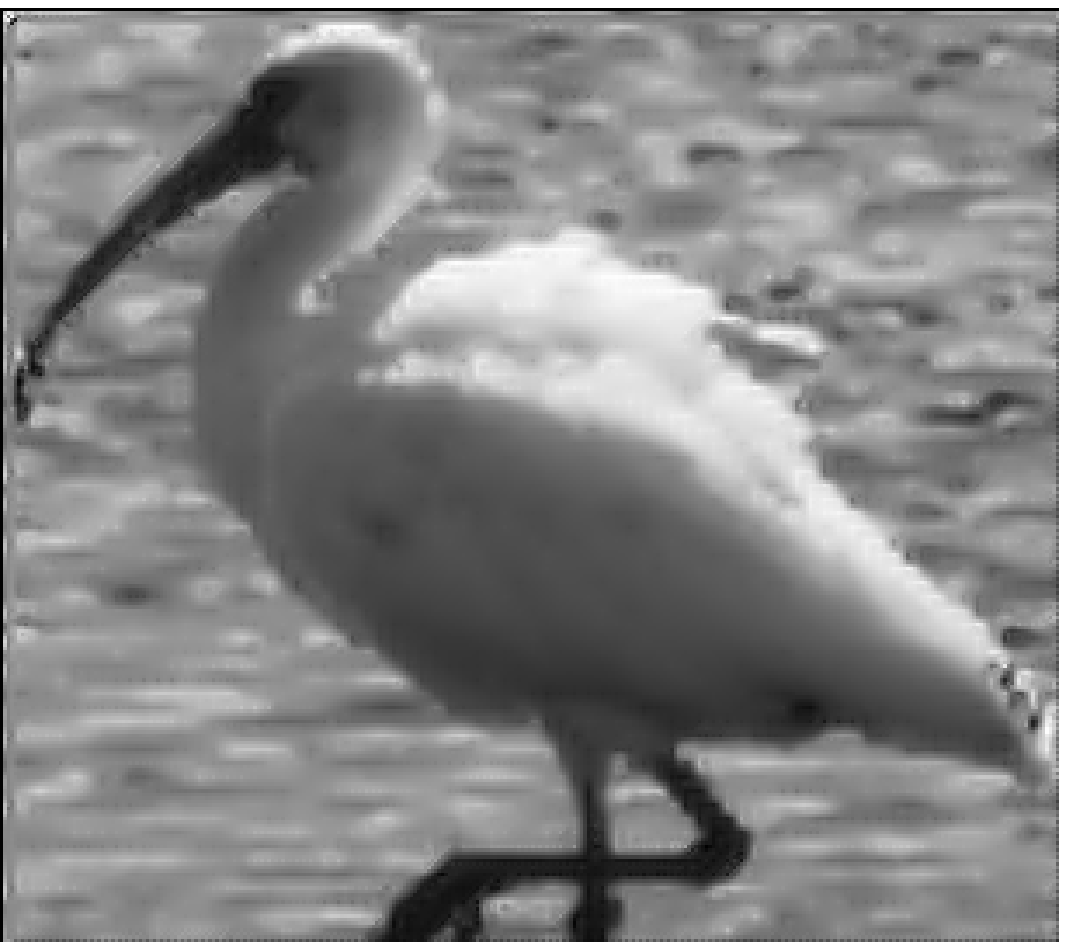}
\includegraphics[width=.45\columnwidth]{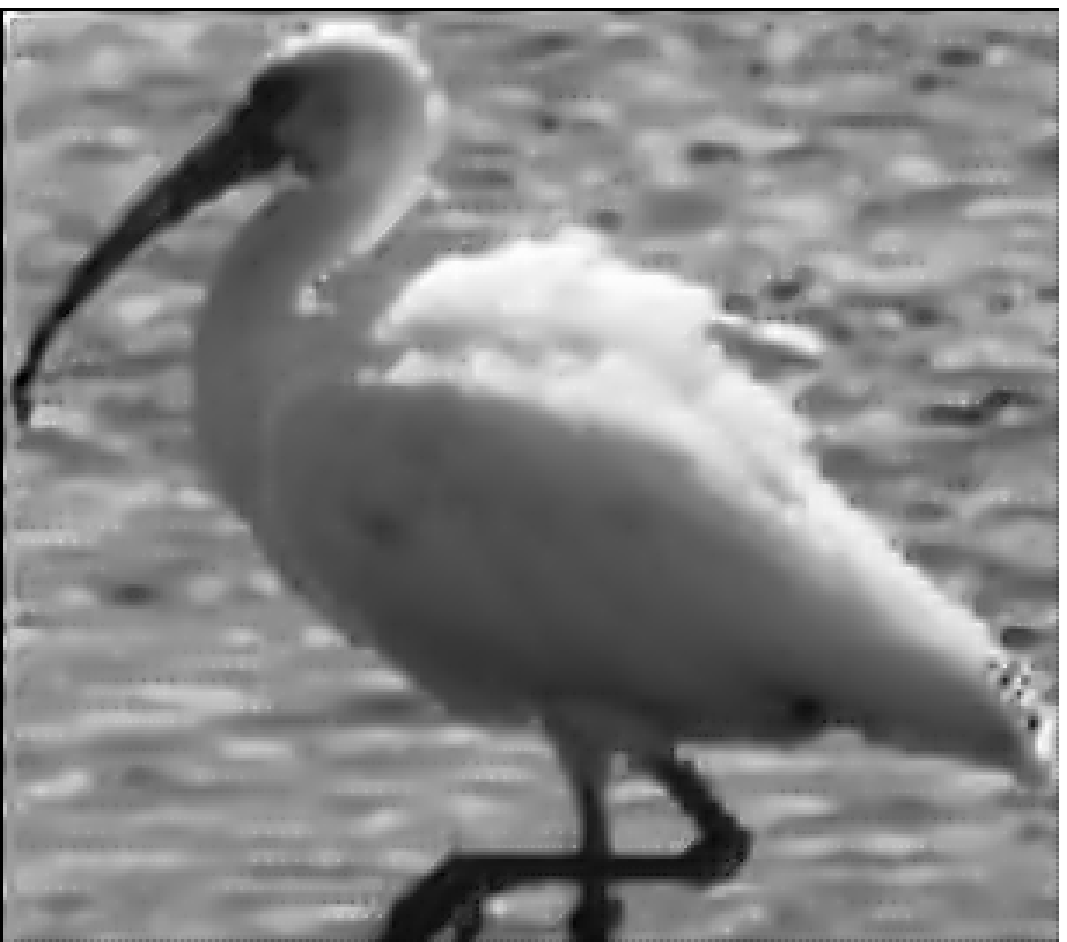}
\caption{Resulting images for image interpolation.  Interpolation to twice as many pixels.  Top left result from the method of Bertalm{\'\i}o \etal, top right result from bicubic interpolation, bottom right result from minimization using the Euler equations, bottom right result from minimization using the $H^1$ gradient.  Notice artifacts are not present when using the $H^1$ gradient (see beak of bird for example).}\label{bird}
\end{figure}

\begin{figure}[h!]
\centering
\includegraphics[width=.45\columnwidth]{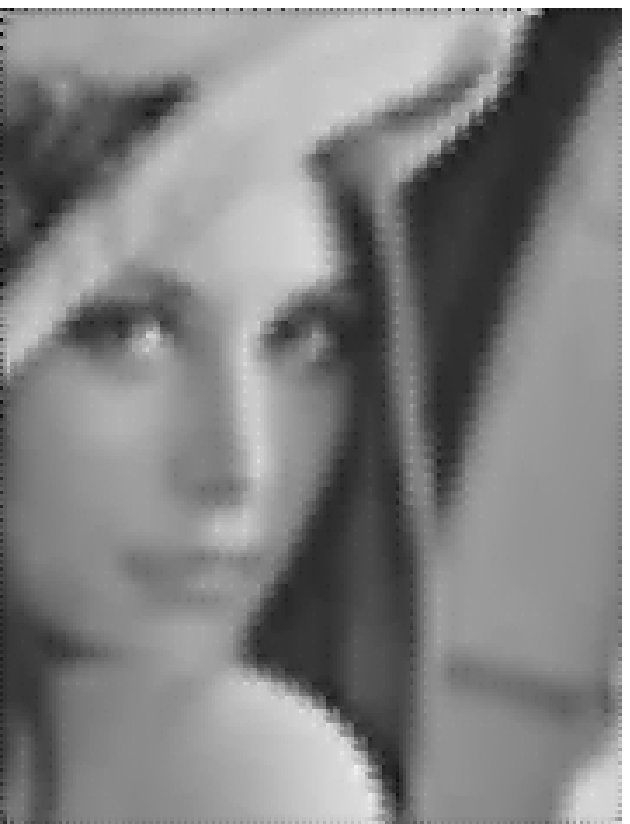}
\includegraphics[width=.45\columnwidth]{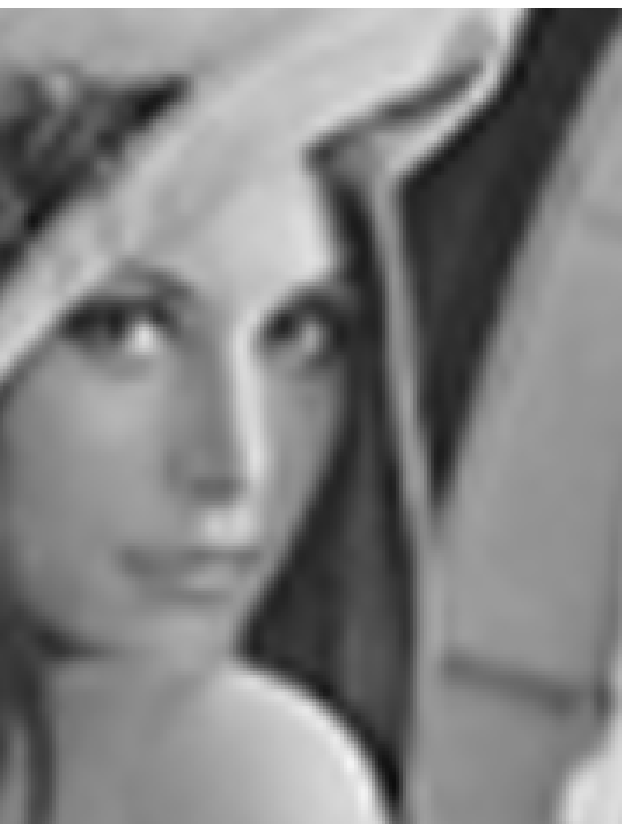}\\
\includegraphics[width=.45\columnwidth]{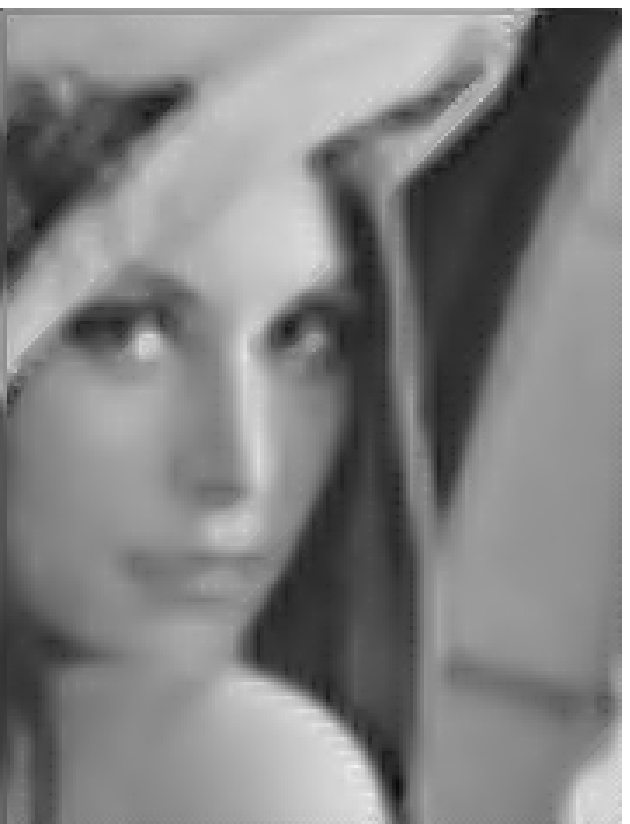}
\includegraphics[width=.45\columnwidth]{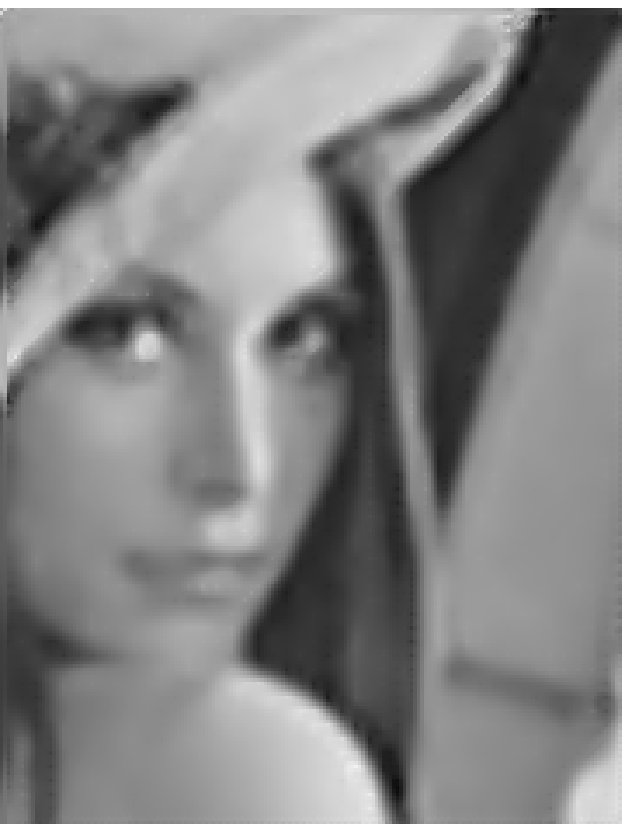}
\caption{Resulting images for image interpolation.  Interpolation to four times as many pixels.  Top left result from the method of Bertalm{\'\i}o \etal\!, top right result from bicubic interpolation, bottom right result from minimization using the Euler equations, bottom right result from minimization using the $H^1$ gradient.}\label{lena}
\end{figure}

We remark here that in practice the assumption that $u$ be $C^1$ continuous on the entire domain $\bar{\Omega}$ is too strong. In image inpainting an edge represents a discontinuity, thus in theory this important case would be excluded from our analysis.  The idea of using a gradient in a higher order Sobolev space is to increase the smoothness of the interpolation, and this is supported by the theory as when one computes a gradient in a higher order Sobolev space, we see that the minimizer should also satisfy further regularity conditions.  In several of our cases, using a Sobolev gradient reduced artifacts associated with image inpainting and image interpolation while preserving edges (see images \ref{bird}, \ref{lena}, and \ref{grsm} and the associated discussions.  However, some care should be taken in practice as over smoothing would result in blurred reconstruction of edges.  This artifact appeared in some or our results (see for example \ref{eye}).  In principle, the analysis of this problem could be studied and carried out in a lower order Sobolev space or in the space of functions of bounded variation which is the traditional space used for an analytical study of problems in image processing.  We expect that the estimates required for such an analysis would be of a similar nature to the estimates required for existence and uniqueness results of the Navier-Stokes equations and hence beyond the scope of this paper.

\section{Summary and final discussion}

In this work we derived a method to solve the digital image inpainting model that is based on the model of the Navier-Stokes equations for an incompressible flow.  We used a variational formulation and a gradient flow to converge to the stationary solution. The flow is based on a gradient coming from a Sobolev norm, instead of the generally used $L^2$ norm, which in practice is associated with preconditioning the Euler-Lagrange equations.  We found that through the use of the preconditioner, we can favorably alter properties of the interpolation by reducing noise and offering a more accurate interpolation of the edges.

Compared to the original Navier-Stokes model proposed by Bertalm{\'\i}o, Bertozi and Sapiro (BBS) \cite{bertalmio2001a}, our algorithm 
displays better convergence properties, with similar qualitative image inpainting results. From the practical point of view of computer implementation, we have eliminated technical aspects related to the use of numerical methods for Navier-Stokes equations (slope limiters, dynamic relaxation for Poisson equation, anisotropic diffusion steps, etc). We also mention that the new algorithm does not require to tune constants, as in non-iterative methods \cite{bornemann}. 
Using a straightforward finite difference implementation, we demonstrated, through various examples for image inpainting and image interpolation, that the algorithm is effective for such problems.

We finally should like to emphasize that the mathematical theoretical framework that we developped in this paper  is of more general interest and could be applied to other variational formulations used in image inpainting, such as the one presented in \cite{masnou1998}. Such new applications of the Sobolev gradient appear as a natural extension of the present work.

\pagebreak
\appendix

\section{Condition number}\label{conditioning}

We give here two precise definitions for the condition number and show that by Sobolev preconditioning, the relative condition number is, as shown in figure \ref{figcond}, improved.

Following the definition given in \cite{trefethen}, the following definitions for the condition number of a function $f$ from an arbitrary normed space $X$ to an arbitrary normed space $Y$ is used.
\begin{definition}
The relative condition number of a differentiable function $f: X \rightarrow Y$ is defined as
\begin{equation*}
\frac{\|f'(x)\|_{L(X,Y)} \|x\|_X}{\|f(x)\|_Y}.
\end{equation*}
\end{definition}
\begin{definition}
The absolute condition number of a function $f : X \rightarrow Y$ is given by
\begin{equation*}
\kappa= \lim_{\delta \rightarrow 0} \text{ sup }_{\|h\|_X \leq \delta}  \frac{\| f(x + h) - f(x) \|}{\|h\|_X}.
\end{equation*}
\end{definition}
$\| f'(x) \|_{L(X,Y)}$ is the norm of $f'(x)$ when viewed as a bounded linear operator from $X$ to $Y$ for a differentiable function $f$.  In case $X=\mathbb{R}^n$ is equipped with the Euclidean metric and $Y=\mathbb{R}$, then $f'(x)$ is just the Jacobian of $f$ at $x$.

Note that if $Y=\mathbb{R}$, then $f'(x)h= \langle h , \nabla_X f(x) \rangle_X$ for each $x \in X$, following the definition of the Sobolev gradient.  Thus $\|f'(x) \|_{L(X , \mathbb{R})} = \ \sup \{ f'(x)h : h \in X, \ \|h\|_X=1 \} = \| \nabla_X f(x)\|_X.$

In the finite dimensional setting, $H^k$ is the space of all $n$-dimensional vectors equipped with the metric $\langle u , v \rangle_{H^k} = \langle u , (I - \Delta)^{k} v \rangle_{\mathbb{R}^n}$.  Here $\Delta$ is a discretization of the Laplacian.  Let $f : H^k \rightarrow \mathbb{R}$, then using the definition of the gradient
\begin{equation*}
f'(x)h = \langle h , \nabla f(x) \rangle_{\mathbb{R}^n} = \langle h , (I - \Delta)^{k} \nabla_{H^k} f(x) \rangle_{\mathbb{R}^n} \ \forall h \in \mathbb{R}^n.
\end{equation*}
Thus
\begin{equation}
\nabla_{H^k} f(x) = (I - \Delta)^{-k} \nabla f(x)
\label{eq-gradHk}
\end{equation}
and
\begin{equation*}
\| \nabla_{H^k} f(x) \|_{H^k} = \langle (I- \Delta)^k (I- \Delta)^{-k} \nabla f(x) , (I - \Delta)^{-k} \nabla f(x) \rangle = \langle \nabla f (x), (I - \Delta)^{-k} \nabla f(x) \rangle.
\end{equation*}
We make two remarks.  First, we see that by equipping the domain of the problem with the $H^k$ norm in place of the Euclidean norm, we precondition the problem so that the absolute condition number of the problem as defined above is reduced by a factor of $| (I - \Delta)^{-1} |^k$.

This effect could be easily illustrated if the Fourier representation of the gradient is used. If the ordinary gradient is sampled on a rectangular, regular grid as:
\begin{equation}
 \nabla f(x_j) = \sum_p \widehat{g} e^{i \lambda_p x_j},
\end{equation}
then, following (\ref{eq-gradHk}), the $H^k$ gradient will have a similar representation (same basis functions) with modified Fourier coefficients:
\begin{equation*}
 \nabla_{H^k} f(x_j) = \sum_p \widehat{G} e^{i \lambda_p x_j},\quad \widehat{G} = \frac{\widehat{g}}{(1+\lambda_p^2)^k}.
\end{equation*}
Since  $k \geq 1$ we infer that
$\| \nabla_{H^k} f(x) \|^2 = \sum \|\widehat{G}\|^2$ is reduced and therefore the condition number is improved. The practical consequence is that diminishing the corresponding wave numbers in the Fourier decomposition results is attenuating the most oscillating modes of the solution. This could explain the smoothing effect of Sobolev gradient methods (see also \cite{garciaripoll}).
Second, using the above reasoning, if $f$ is real valued and $X=H^k$, then $\| f'(x) \|_{L(X,Y)} = \| \nabla_X f(x) \|_X$ and hence the relative condition number is given by
\begin{equation*}
\frac{\langle \nabla f(x), (I - \Delta)^{-k} \nabla f(x) \rangle \|x\|_{H^k}}{|f(x)|}.
\end{equation*}

In the examples we present, we normalize all images so that the maximum pixel value is one and thus use the relative condition number as a measure of how well conditioned the problem is.

In figure \ref{figcond}, we give a plot of the relative condition number for the case corresponding to image \ref{grsm}.  We compare the relative condition number of the problem with and without preconditioning.  The relative condition number is lowest for $H^3$, achieving a reduction of a factor of 4 compared with the Lagrange equations.

\begin{figure}[h!]
\centering
\includegraphics[width=1\columnwidth]{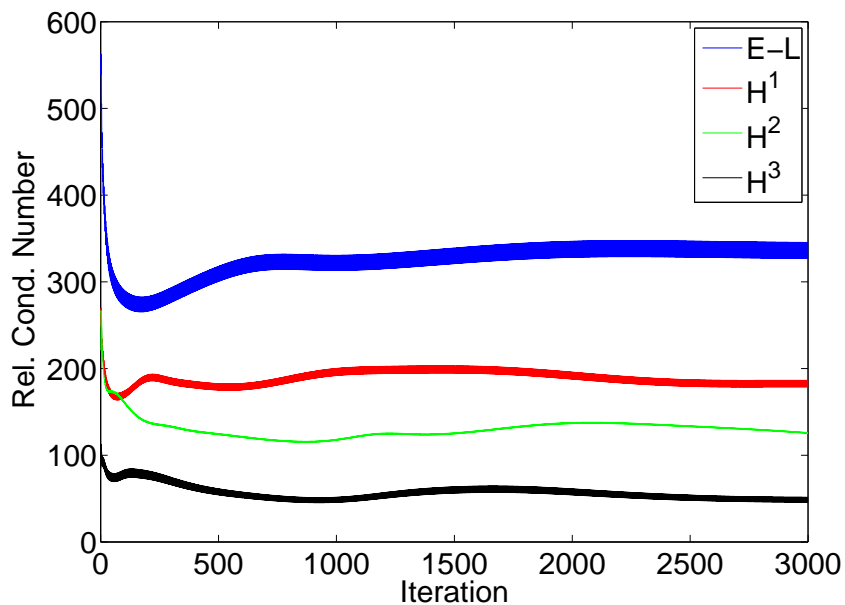}
\caption{Plot of the relative condition number versus iteration corresponding to the inpainting of figure \ref{grsm}.  E-L' corresponds to the minimization of the energy with the Euler-Lagrange equation  (no preconditioning), and $H^1$, $H^2$, and $H^3$ correspond to the minimization using the $H^1$, $H^2$, or $H^3$ gradients.}\label{figcond}
\end{figure}

Another possibility to deal with the condition number issues is to reformulate this problem as a first order system.  This method is currently under investigation in \cite{renkans} and could be possibly applied  to the inpainting problem.

\section*{Acknowledgments}
We thank Dr. Santiago Betelu for suggesting this problem.  We would also like to acknowledge J. W. Neuberger, S. Masnou, and F. Hecht for stimulating discussions and useful suggestions.

\end{document}